\documentclass[12pt]{amsart}
\usepackage{amsmath,amssymb,amsthm,hyperref,color,euscript,enumerate}

\newcommand{\cds}{\cdots}

\newcommand{\1}{{\bf 1}}

\newcommand{\vsb}{\vspace{2mm}}

\newcommand{\q}{\quad}

\newcommand{\maru}[1]{{\ooalign{\hfil#1\/\hfil\crcr
\raise.167ex\hbox{\mathhexbox20D}}}}

\newcommand{\ruby}[2]{%
 \leavevmode
 \setbox0=\hbox{#1}%
 \setbox1=\hbox{\tiny #2}%
 \ifdim\wd0>\wd1 \dimen0=\wd0 \end{lemma}se \dimen0=\wd1 \fi
 \hbox{%
   \kanjiskip=0pt plus 2fil
   \xkanjiskip=0pt plus 2fil
   \vbox{%
     \hbox to \dimen0{%
       \tiny \hfil#2\hfil}%
     \nointerlineskip
     \hbox to \dimen0{\mathstrut\hfil#1\hfil}}}}

\DeclareMathOperator*{\fusion}{\boxtimes}

\newcommand{\Z}{\mathbb{Z}}
\newcommand{\C}{\mathbb{C}}
\newcommand{\R}{\mathbb{R}}

\newcommand{\RM}{\mathrm{RM}}

\newcommand{\F}{\mathbb{F}}

\newcommand{\Sym}{{\rm Sym}}

\newcommand{\vir}{\mathrm{Vir}}
\newcommand{\aut}{\mathrm{Aut}}
\newcommand{\Aut}{\mathrm{Aut}}
\newcommand{\wt}{\mathrm{wt}}

\newcommand{\be}{\beta}
\newcommand{\al}{\alpha}

\newcommand{\Span}{\mathrm{Span}}
\newcommand{\EuD}{\EuScript{D}}

\makeatletter \@addtoreset{equation}{section}

\theoremstyle{plain}
\newtheorem{theorem}{Theorem}[section]
\newtheorem{proposition}[theorem]{Proposition}
\newtheorem{lemma}[theorem]{Lemma}
\newtheorem{corollary}[theorem]{Corollary}

\theoremstyle{definition}
\newtheorem{definition}[theorem]{Definition}

\newtheorem{notation}[theorem]{Notation}
\theoremstyle{remark}
\newtheorem{remark}[theorem]{Remark}
\numberwithin{equation}{section}

\title[Holomorphic framed VOAs]{Quadratic spaces and
holomorphic framed vertex operator algebras of central charge 24}
\author{Ching Hung Lam} %
  \address[C. H. Lam] {Institute of Mathematics, Academia Sinica, Taipei 10617, Taiwan}
  \email{chlam@math.sinica.edu.tw}
\author[H. Shimakura]{Hiroki Shimakura}%
\address[H. Shimakura]{Department of Mathematics,
Aichi University of Education,
1 Hirosawa, Igaya-cho, Kariya-city, Aichi, 448-8542 Japan}%
\email {shima@auecc.aichi-edu.ac.jp}%
\date{}
\thanks{C.\,H. Lam was partially supported by NSC grant
  97-2115-M-006-015-MY3 and National Center for Theoretical Sciences, Taiwan.}
\thanks{H.\ Shimakura was partially supported by Grants-in-Aid for Scientific Research (No. 20549004, No. 23540013) and by Excellent Young Researcher Overseas Visit Program, Japan Society for the Promotion of Science.}

\newcommand{\sfr}[2]{\leavevmode\kern-.1em
  \raise.5ex\hbox{\the\scriptfont0 #1}\kern-.1em
  /\kern-.15em\lower.25ex\hbox{\the\scriptfont0 #2}}

\newcommand{\shf}{\sfr{1}{2}}

\begin{document}

\begin{abstract}
In 1993, Schellekens \cite{Sc93}
obtained a list of possible $71$ Lie algebras of holomorphic vertex
operator algebras with central charge $24$. However, not all cases are
known to exist. The aim of this article is to construct new
holomorphic vertex operator algebras using the theory of framed vertex
operator algebras  and to determine the Lie algebra structures of their
weight one subspaces. In particular, we  study holomorphic framed
vertex operator algebras associated to  subcodes of the triply even
codes $\RM(1,4)^3$ and $\RM(1,4)\oplus \EuD(d_{16}^+)$ of length $48$.
These vertex operator algebras correspond to the holomorphic simple
current extensions of the lattice type vertex operator algebras
$(V_{\sqrt{2}E_8}^+)^{\otimes 3}$ and $V_{\sqrt{2}E_8}^+\otimes
V_{\sqrt{2}D_{16}^+}^+$. We determine such extensions using a
quadratic space structure on the set of all irreducible modules $R(W)$
of  $W$ when  $W= (V_{\sqrt{2}E_8}^+)^{\otimes 3}$  or
$V_{\sqrt{2}E_8}^+\otimes V_{\sqrt{2}D_{16}^+}^+$ \cite{Sh2,Sh6}.
As our main results, we construct seven new holomorphic vertex
operator algebras of central charge $24$ in Schellekens' list and
obtain a complete list of all Lie algebra structures associated to the
weight one subspaces
of holomorphic framed vertex operator algebras of central charge $24$.
\end{abstract}
\maketitle


\section*{Introduction}
The classification of even unimodular lattices of rank $24$ is  one of fundamental
results in lattice theory; there are exactly $24$ such lattices and each lattice is
uniquely determined by its root system -- the set of norm $2$ vectors. Since
there are many analogies between lattices and vertex operator algebras (VOAs),
it is natural to consider the corresponding classification problem for holomorphic
VOAs of central charge $24$.
In 1993 Schellekens \cite{Sc93} obtained a list of possible $71$ Lie algebras of holomorphic vertex operator algebras with central charge $24$.
However, not all cases are constructed explicitly and known to exist. In fact, only the
cases for VOAs associated to even unimodular lattices and their $\Z_2$-orbifolds  are well-studied
(\cite{FLM,DGM}).

Framed VOA is another class of  well-studied VOAs (\cite{DGH,Mi04,LY}). Roughly
speaking, a framed VOA is a simple VOA which contains a full subalgebra $T$,
called a Virasoro frame, isomorphic to a tensor product of copies of the Virasoro
VOA $L(1/2, 0)$. Such a VOA is rational and $C_2$-cofinite, and its structure is
mainly determined by certain combinatorial objects. Therefore, it is natural to
consider the classification of holomorphic framed VOAs of central charge $24$.

It was shown in \cite{LY} that a binary code $D$ of
length $16k$ can be realized as the $1/16$-code of a holomorphic framed VOA of
central charge $8k$ if and only if $D$ is triple even (i.e., $\wt(\al)\equiv
0\mod 8$ for all $\al \in D$) and the all-one vector $(1,\cdots,1)\in D$.
Therefore, the classification of holomorphic framed VOAs of central charge $24$
is almost equivalent to the problem of classifying all triply even codes of
length $48$ and the study of possible VOA structures associated to each triply
even code.

In \cite{BM}, triply even codes of length $48$ are classified: any triply even code of length $48$ is a subcode of one of the following:
\begin{enumerate}[(1)]
\item an  extended doubling $\EuD(E)$ for some doubly even code $E$ of length $24$ (see Definition \ref{double});
\item the  $9$-dimensional exceptional triply even code $D^{ex}$ of length $48$;
\item the  direct sum of three copies of the Reed-Muller code $\RM(1,4)$;
\item the  direct sum of ${\rm RM}(1,4)$ and $\EuD(d_{16}^+)$, where $d_{16}^+$ is the unique indecomposable doubly even self-dual code of length $16$.
\end{enumerate}
It was shown in \cite{Lam} that if the $1/16$-code is a subcode of  an extended doubling (Case 1), then the framed VOA is isomorphic to a lattice VOA or its $\Z_2$-orbifold.
Moreover, holomorphic framed VOAs associated to subcodes of $D^{ex}$ (Case 2) have been constructed and studied in
\cite{Lam}. The Lie algebras associated to their weight one subspaces are also
determined.
In particular, $10$ new holomorphic framed VOAs are found mathematically.
In order to complete the classification of holomorphic framed VOAs, it is very important to construct and study the holomorphic framed VOAs associated to the subcodes in Cases 3 and 4.



In this article, we will study the framed VOAs associated to the triply even
codes isomorphic to subcodes of $\RM(1,4)^3$ and $\RM(1,4)\oplus \EuD(d_{16}^+)$.
These VOAs correspond to the holomorphic simple current extensions of
$(V_{\sqrt{2}E_8}^+)^{\otimes 3}$ and $V_{\sqrt{2}E_8}^+\otimes
V_{\sqrt{2}D_{16}^+}^+$. It was proved in \cite{Sh2} that the  set of all
irreducible modules $R(V_{\sqrt{2}E_8}^+)$ and $R(V_{\sqrt{2}D_{16}^+}^+) $ of
$V_{\sqrt{2}E_8}^+$ and $V_{\sqrt{2}D_{16}^+}^+$ have the structures of some
quadratic spaces and $\aut(V_{\sqrt{2}E_8}^+)$ and
$\aut(V_{\sqrt{2}D_{16}^+}^+)$ act naturally on  $R(V_{\sqrt{2}E_8}^+)$ and
$R(V_{\sqrt{2}D_{16}^+}^+)$, respectively. By using these quadratic spaces,
we determine all holomorphic extensions of $(V_{\sqrt{2}E_8}^+)^{\otimes 3}$
and $V_{\sqrt{2}E_8}^+\otimes V_{\sqrt{2}D_{16}^+}^+$, up to conjugation, and
compute the associated Lie algebra structures of the corresponding weight one
spaces. As a consequence, we construct
 seven new holomorphic framed VOAs having Lie algebras $ D_{4,4}(A_{2,2})^4$,
$C_{4,2}(A_{4,2})^2$, $E_{6,2}C_{5,1}A_{5,1}$, $C_{8,1}(F_{4,1})^2$,
$E_{7,2}B_{5,1}F_{4,1}$, $E_{8,2}B_{8,1}$ and $A_{8,2}F_{4,2}$. In addition, we
obtain a complete list of all Lie algebra structures for the weight one
subspaces of holomorphic framed VOAs of central charge $24$. The main result is
as follows.

\begin{theorem} \label{thm:Lieframed}
Let $V$ be a holomorphic framed VOA of central charge
$24$. Then one of the following holds:
\begin{enumerate}[{\rm (1)}]
\item $V$ is isomorphic to a lattice VOA $V_N$ or its $\Z_2$-orbifold $\tilde{V}_N$ for some even unimodular
lattice $N$;
\item the weight one subspace $V_1$ is isomorphic to one of the Lie algebras
 in Table \ref{Lieframed}.
 \end{enumerate}
 Moreover, for each Lie algebra $\mathfrak{L}$ in Table \ref{Lieframed},
 there exists a holomorphic framed VOA $U$ whose weight one subspace $U_1$
 is isomorphic to $\mathfrak{L}$.
\end{theorem}

\begin{table}[bht]
\caption{ \bf Lie algebras for holomorphic framed VOAs of $c=24$}
\label{Lieframed}
\begin{center}
\begin{tabular}{|c|c|c|}
\hline
No. in \cite{Sc93}& Dimension & Lie algebra \cr
\hline \hline
$7$  & $48$  & $({A_{3,4}})^3A_{1,2}$\cr
 \hline
$10$ &  $48$ & $D_{5,8}A_{1,2}$ \cr   \hline
$13$ & $60$ & $D_{4,4}(A_{2,2})^4$ \cr \hline
 $18$  &  $72$  & $A_{7,4}({A_{1,1}})^3$\cr
 \hline
$19$ &  $72$  & $  D_{5,4}C_{3,2}({A_{1,1}})^2$\cr
\hline
$22$ & $84$ & $C_{4,2}(A_{4,2})^2$ \cr \hline
$26$ &  $96$  &
  $ ({A_{5,2}})^2 C_{2,1}({A_{2,1}})^2$\cr
\hline
$33$ &  $120$ & $ A_{7,2}  ({C_{3,1}})^2A_{3,1}$\cr \hline
$35$ &  $120$  & $C_{7,2}A_{3,1}$\cr\hline
$36$ &  $132$ &  $A_{8,2}F_{4,2}$\cr \hline
$40$ &  $144$ & $  A_{9,2}A_{4,1}B_{3,1}$\cr \hline
$44$ & $168$&  $E_{6,2}C_{5,1}A_{5,1}$\cr \hline
$48$ & $192$ & $ ({C_{6,1}})^2B_{4,1}$\cr \hline
$52$ & $240$ & $C_{8,1}(F_{4,1})^2$ \cr \hline
$53$ & $240$ & $E_{7,2}B_{5,1}F_{4,1}$ \cr \hline
$56$ &  $288$ & $ C_{10,1}B_{6,1}$\cr \hline
$62$ & $384$ & $E_{8,2}B_{8,1}$       \cr \hline
\end{tabular}
\end{center}
\end{table}

The organization of this article is as follows. In Section 1, we recall some
basic facts about quadratic spaces and orthogonal groups. We also review the
notions of VOAs, modules and intertwining operators.  In Section 2,  the notion
and several important properties of framed VOAs will be reviewed. We will also
recall the classification of triply even codes of length $48$ from \cite{BM}.
In Section 3, the representation theory of the lattice type VOAs
$V_{\sqrt{2}E_8}^+$ and $V_{\sqrt{2}D_{16}^+}^+$ is reviewed. In particular,
the quadratic spaces associated to the set of their irreducible modules are
discussed. In Section 4, we will study the holomorphic simple current
extensions of $(V_{\sqrt{2}E_8}^+)^{\otimes 3}$. The main result is a complete
description of all maximal totally singular subspaces of
$R(V_{\sqrt{2}E_8}^+)^{\otimes 3}$, up to the action of
$\aut((V_{\sqrt{2}E_8}^+)^{\otimes 3})$. The Lie algebra structures of the
weight one subspaces of the corresponding framed VOAs associated to these
quadratic spaces will also be computed.   In Section 5, we study the
holomorphic simple current extensions of  $V_{\sqrt{2}E_8}^+\otimes
V_{\sqrt{2}D_{16}^+}^+$. We also obtain a complete classification of all maximal
totally singular subspaces of $R(V_{\sqrt{2}E_8}^+)\oplus
R(V_{\sqrt{2}D_{16}^+}^+)$, up to the action of $\aut( V_{\sqrt{2}E_8}^+\otimes
V_{\sqrt{2}D_{16}^+}^+)$. Again, the Lie algebras associated to the weight one
subspaces of the corresponding framed VOAs will be computed.


\section{Preliminary}

\begin{center}
{\bf Notations}
\begin{small}
\begin{tabular}{ll}
\ \\
$\langle\ , \ \rangle$& The symplectic form on a quadratic space $(R,q)$ or on $(R^k,q^k)$.\\
$\boxtimes$& The fusion rules for a VOA.\\
$\perp $& The orthogonal direct sum of subspaces in a quadratic space.\\
$A^\perp$& The orthogonal complement of a subspace $A$ in a quadratic space.\\
$D_{16}^+$& The indecomposable even unimodular lattice of rank $16$,\\
& whose root lattice is $D_{16}$.\\
$E_8$& The $E_8$ root lattice, even unimodular lattice of rank $8$.\\
$g\circ M$& The conjugate of a module $M$ for a VOA by an automorphism $g$.\\
$g\circ [M]$& The isomorphism class of $g\circ M$.\\
$[M]$& The isomorphism class of a module $M$ for a VOA.\\
$N(L)$& The even unimodular lattice of rank $24$ whose root lattice is isomorphic to $L$.\\
$\rho_i$& The $i$-th coordinate projection of a direct sum of quadratic spaces.\\
$q_V$& The quadratic form on $R(V)$ defined by\\
& $q_V([M])=0$ and $1$ if $M$ is $\Z$-graded and $(\Z+1/2)$-graded, respectively.\\
$O(R,q)$& The orthogonal group of $(R,q)$.\\
$\Sym_n$& The symmetric group of degree $n$.\\
$\mathcal{S}$& A maximal totally singular subspace of $R^k$ or $R(V)^k$.\\
$\mathcal{S}(m,k_1,k_2,\varepsilon)$ & The maximal totally singular subspace of $R^3$ defined in Theorem \ref{TClassify}.\\
$\mathcal{S}(m,k_1,k_2)$ & The maximal totally singular subspace of $R^3$ defined in Theorem \ref{TClassify2}.\\
$S(R)$& The set of singular vectors in $R$\\
$S(R)^\times$& The set of non-zero singular vectors in $R$\\
$\overline{S(R)}$& The set of non-singular vectors in $R$\\
${\rm Stab}_G(A)$& The setwise stabilizer of $A$ in a group $G$.\\
${\rm Stab^{pt}}_G(A)$& The pointwise stabilizer of $A$ in a group $G$.\\

$(R,q)$& A non-singular quadratic space $R$ of plus type with quadratic form $q$ over $\F_2$.\\
$(R^k,q^k)$& The orthogonal direct sum of $k$ copies of $(R,q)$.\\
$R(V)$& The set of all isomorphism classes of irreducible modules for a VOA $V$.\\
$V$& A simple rational $C_2$-cofinite self-dual VOA of CFT type, or $V=V_{\sqrt2E_8}^+$.\\
$V_L$& The lattice VOA associated with even lattice $L$.\\
$V_L^+$& The fixed point subVOA of $V_L$ with respect to a lift of the $-1$-isometry of $L$.\\
$\tilde{V}_L$& $\Z_2$-orbifold of $V_L$.\\
$V^k$& The tensor product of $k$ copies of a VOA $V$.\\
$\mathfrak{V}(\mathcal{S})$& The holomorphic VOA associated to a maximal totally singular subspace $\mathcal{S}$.\\
$\mathfrak{V}(\mathcal{T})$& The module associated to a subset $\mathcal{T}$ of $R(V)^k$ or $R(X)\oplus R(V)$.\\
$X$& $X=V_{\sqrt2D_{16}^+}^+$.\\
\end{tabular}
\end{small}
\end{center}

\subsection{Quadratic spaces and orthogonal groups}
Let us recall fundamental facts on quadratic spaces over $\F_2$ and orthogonal groups (cf.\ \cite{Ch}).

Let $R$ be a $2m$-dimensional vector space over $\F_2$.
A form $\langle\cdot,\cdot\rangle:R\times R\to \F_2$ is said to be {\it symplectic} if it is a symmetric bilinear form with $\langle a,a\rangle=0$ for all $a\in R$.
A map $q:R\to\F_2$ is called a {\it quadratic} form associated to $\langle\cdot,\cdot\rangle$ if $\langle a,b\rangle=q(a+b)+q(a)+q(b)$ for all $a,b\in R$.
Let $q$ be a quadratic form.
Then the pair $(R,q)$ is called a {\it quadratic space}, and it is said to be {\it non-singular} if $R^\perp=\{a\in R\mid \langle a,R\rangle=0\}=0$.
A vector $a\in R$ is said to be {\it singular} and {\it non-singular} if $q(a)=0$ and $q(a)=1$, respectively.
Let $S(R)$, $S(R)^\times$, and $\overline{S(R)}$ denote the sets of all singular vectors in $R$, of all non-zero singular vectors in $R$, and of all non-singular vectors in $R$, respectively.
A subspace $S$ of $R$ is said to be {\it totally singular} if any vector in $S$ is singular.
A quadratic form $q$ is said to be {\it of plus type} and {\it of minus type} if the dimension of a maximal totally singular subspace of $(R,q)$ is $m$ and $m-1$, respectively.
Let $O(R,q)$ denote the orthogonal group of $(R,q)$, the subgroup of ${\rm GL}(R)$ preserving $q$.
The following lemmas are well-known.

\begin{lemma}\label{LNum} Let $(R,q)$ be a non-singular $2m$-dimensional quadratic space of $\varepsilon$ type over $\F_2$, where $\varepsilon\in\{\pm\}$.
\begin{enumerate}
\item If $\varepsilon=+$ then $R$ has $2^{2m-1}+2^{m-1}-1$ non-zero singular vectors and $2^{2m-1}-2^{m-1}$ non-singular vectors.
\item If $\varepsilon=-$ then $R$ has $2^{2m-1}-2^{m-1}-1$ non-zero singular vectors and $2^{2m-1}+2^{m-1}$ non-singular vectors.
\end{enumerate}
\end{lemma}

\begin{lemma}\label{LO} {\rm (cf.\ \cite{Ch})} Let $(R,q)$ be a non-singular $2m$-dimensional quadratic space of plus type over $\F_2$.
Let $k\in\{1,2,\dots, m\}$.
\begin{enumerate}
\item The orthogonal group $O(R,q)$ is transitive on the set of all $k$-dimensional totally singular subspaces of $R$.
\item The stabilizer of a $k$-dimensional totally singular subspace in $O(R,q)$ has the shape $2^{{\binom{k}{2}}+k(2m-2k)}.({\rm SL}_k(2)\times O^+({2m-2k},2))$.
\item Let $U$ be a subspace of $R$ such that $\langle U,U\rangle=0$ and $\dim U\ge 2$.
Then $U$ contains a non-zero singular vector.
\end{enumerate}
\end{lemma}

\subsection{Vertex operator algebras}
Throughout this article, all VOAs are defined over the field $\C$ of complex
numbers unless otherwise stated.  We recall the notions of vertex operator
algebras (VOAs) and modules from \cite{Bo,FLM,FHL}.

A {\it vertex operator algebra} (VOA) $(V,Y,\1,\omega)$ is a $\Z_{\ge0}$-graded
 vector space $V=\oplus_{m\in\Z_{\ge0}}V_m$ equipped with a linear map

$$Y(a,z)=\sum_{i\in\Z}a_{(i)}z^{-i-1}\in ({\rm End}(V))[[z,z^{-1}]],\quad a\in V$$
and the {\it vacuum vector} $\1$ and the {\it conformal element} $\omega$
satisfying a number of conditions (\cite{Bo,FLM}). We often denote it by $V$ or
$(V,Y)$.

Two VOAs $(V,Y,\1,\omega)$ and $(V^\prime,Y',\1',\omega')$ are said to be {\it
isomorphic} if there exists a linear isomorphism $g$ from $V$ to $V^\prime$
such that $$ g\omega=\omega'\quad {\rm and}\quad gY(v,z)=Y'(gv,z)g\quad
\text{ for all } v\in V.$$ When $V=V'$,  such a linear isomorphism is called an
{\it automorphism}. The group of all automorphisms of $V$ is called the {\it
automorphism group} of $V$ and is denoted by $\Aut(V)$.

A {\it vertex operator subalgebra} ( or a {\it subVOA} ) is a graded subspace of
$V$ which has a structure of a VOA such that the operations and its grading
agree with the restriction of those of $V$ and that they share the vacuum vector.
When they also share the conformal element, we will call it a {\it full subVOA}.

An (ordinary) module $(M,Y_M)$ for a VOA $V$ is a $\C$-graded vector space $M=\oplus_{m\in\C} M_{m}$ equipped with a linear map
$$Y_M(a,z)=\sum_{i\in\Z}a_{(i)}z^{-i-1}\in ({\rm End}(M))[[z,z^{-1}]],\quad a\in V$$
satisfying a number of conditions (\cite{FHL}).
We often denote it by $M$ and its isomorphism class by $[M]$.
The {\it weight} of a homogeneous vector $v\in M_k$ is $k$.
If $M$ is irreducible then $M=\oplus_{m\in\Z_{\ge0}}M_{h+m}$ ($h\in\C, M_h\neq0$), where $h$ is the {\it lowest weight} of $M$.

A VOA $V$ is said to be {\it of CFT type} if $V_0=\C\1$, is said to be {\it
rational} if any module is completely reducible, and is said to be {\it
$C_2$-cofinite} if $V/{\rm Span}_\C\{ a_{(-2)}b\mid a,b\in V\}$ is
finite-dimensional. Note that any module is ordinary if $V$ is $C_2$-cofinite.
A VOA is said to be {\it holomorphic} if it is the only irreducible module up
to isomorphism. A module $M$ is said to be {\it self-dual} if its
contragradient module (cf.\ \cite{FHL}) is isomorphic to itself. Let $R(V)$
denote the set of all isomorphism classes of irreducible $V$-modules. Note that
if $V$ is rational or $C_2$-cofinite then $|R(V)|<\infty$.

Let $M_a,M_b,M_c$ be modules for a simple rational $C_2$-cofinite VOA $V$.
An intertwining operator of type $M_a\times M_b\to M_c$ is a linear map $M_a\mapsto ({\rm Hom}(M_b,M_c))\{z\}$ satisfying a number of conditions (\cite{FHL}).
Let $N_{M_a,M_b}^{M_c}$ denote the dimension of the space of all intertwining operators of type $M_a\times M_b\to M_c$, which is called the {\it fusion rule}.
Since $V$ is $C_2$-cofinite, the fusion rules are finite (\cite{ABD}).
By the definition of the fusion rules, $N_{M_a,M_b}^{M_c}=N_{M_a',M_b'}^{M_c'}$ if $M_x\cong M_x'$ as $V$-modules for all $x=a,b,c$.
Hence, the fusion rules are determined by the isomorphism classes of $V$-modules.
For convenience, we use the following expression $$[M_a]\boxtimes [M_b]=\bigoplus_{[M]\in R(V)} N_{M_a,M_b}^{M}[M],$$
which is also called the fusion rule.

Let $M=(M,Y_M)$ be a module for a VOA $V$.
For $g\in\Aut(V)$, let $g\circ M=(M,Y_{g\circ M})$ denote the $V$-module defined by $Y_{g\circ M}(v,z)=Y_M(g^{-1}v,z)$.
If $M\cong M'$ as $V$-modules, then $g\circ M\cong g\circ M'$ as $V$-modules.
Hence we use the notation $g\circ[M]$ for the isomorphism class of $g\circ M$.
If $M$ is irreducible, then so is $g\circ M$.
Hence $\Aut(V)$ acts on $R(V)$.
We often identify modules of a VOA with their respective isomorphism classes and use the same notation for the conjugates.

The theory of tensor products of VOAs was established in \cite{FHL}.
For a positive integer $k$, let $V^k$ denote the tensor product of $k$ copies of $V$.
Later, we use the following lemma.
\begin{lemma}\label{LemFHL} {\rm (\cite[Section 4.7]{FHL}, \cite{DMZ})} Let $V$ be a simple rational $C_2$-cofinite VOA of CFT type.
Then $$R({V^k})=\{\otimes_{i=1}^k W_i\mid W_i\in R(V)\},$$ and for $\otimes_{i=1}^k W_{i,a}, \otimes _{i=1}^kW_{i,b}\in R(V^k)$, the following fusion rule holds:
$$\left(\otimes_{i=1}^k W_{i,a}\right)\boxtimes \left(\otimes _{i=1}^kW_{i,b}\right)
=\otimes_{i=1}^k (W_{i,a}\boxtimes W_{i,b}).$$
\end{lemma}

Let $V(0)$ be a simple VOA.
An irreducible $V(0)$-module $M_a$ is called a {\it simple current module} if for any irreducible $V(0)$-module $M_b$, there exists a unique irreducible $V(0)$-module $M_c$ satisfying the fusion rule $[M_a]\boxtimes [M_b]=[M_c]$.
A simple VOA $V$ is called a {\it simple current extension} of $V(0)$ graded by a finite abelian group $E$ if $V$ is the direct sum of non-isomorphic irreducible simple current $V(0)$-modules $\{V(\alpha)\mid \alpha\in E\}$ and the fusion rule
$[V(\alpha)]\boxtimes [V(\beta)]=[V(\alpha+\beta)]$ holds for all $\alpha,\beta\in E$.

\begin{lemma}\label{LSY} {\rm \cite[Lemma 3.14]{SY}} Let $V=\oplus_{\alpha\in E}V(\alpha)$ be a simple current extension of a simple VOA $V(0)$ graded by a finite abelian group $E$.
Let $g$ be an automorphism of $V(0)$.
Then there exists a simple current extension of $V(0)$ such that it is isomorphic to $V$ as VOAs and is isomorphic to $\oplus _{\alpha\in E}\ g\circ V(\alpha)$ as $V(0)$-modules.
\end{lemma}

\subsection{Lattice VOAs and $\Z_2$-orbifolds}
Let $L$ be an even unimodular lattice and let $V_L$ be the lattice VOA associated with $L$ (\cite{Bo,FLM}).
Then $V_L$ is holomorphic (\cite{Do}).
Let $\theta\in \aut(V_L)$ be a lift of $-1\in\aut(L)$ and let $V_L^+$ denote the subVOA of $V_L$ consisting of vectors in $V_L$ fixed by $\theta$.
Let $V_L^T$ be a unique $\theta$-twisted module for $V_L$ and $V_L^{T,+}$ the irreducible $V_L^+$-submodule of $V_L^T$ with integral weights. 
Let $$ \tilde{V}_L = V_L^+ \oplus V_L^{T,+}.$$ It is known that $ \tilde{V}_L$
has a unique holomorphic VOA structure by extending its
$V_L^+$-module structure (see \cite{FLM,DGM}).
The VOA  $\tilde{V}_L$ is often called the {\it $\Z_2$-orbifold} of $V_L$ .

\begin{remark} Assume that the rank of $L$ is $24$.
Then any even unimodular lattice $L$ has an orthogonal basis of norm $4$ (\cite{HK}).
Hence both $V_L$ and $\tilde{V}_L$ are framed (\cite{DMZ}).
\end{remark}

\section{Framed vertex operator algebras}

In this section, we review the notion of framed VOAs from \cite{DGH,Mi04}.

Let $\vir=\oplus_{n\in\Z}\C L_n \oplus \C \mathbf{c}$ be the Virasoro algebra.
That means  the $L_n$ satisfy the commutator relations:
\[
\begin{split}
 [L_m,L_n]=(m-n)L_{m+n}+ \frac{1}{12}(m^3-m) \delta_{m+n,0} \mathbf{c},\quad\text{ and }\quad
 [L_m,\mathbf{c}]=0.
\end{split}
\]
For any $c,h \in \C$, we will denote by $L(c,h)$ the irreducible highest weight
module of $\vir$ with central charge $c$ and highest weight $h$.
It is shown by \cite{FZ} that $L(c,0)$ has a natural VOA structure. We call it the
\emph{simple Virasoro VOA} with central charge $c$.

\begin{definition}\label{df:3.1}
Let $V=\bigoplus_{n=0}^\infty V_n$ be a VOA. An element $e\in V_2$ is  called
an {\it Ising vector}
if the
subalgebra $\vir(e)$ generated by $e$ is isomorphic to $L( \shf,0)$ and $e$ is
the conformal element of $\vir(e)$. Two Ising vectors $u,v\in V$ are said to be
{\it orthogonal} if $[Y(u,z_1), Y(v,z_2)]=0$.
\end{definition}

\begin{remark}
It is well-known that $L(\shf,0)$ is rational, i.e., all  $L(\shf, 0)$-modules are
completely reducible, and has only three inequivalent irreducible modules
$L(\shf,0)$, $L(\shf,\shf)$ and $L(\shf,\sfr{1}{16})$. The fusion rules of
$L(\shf,0)$-modules are computed in \cite{DMZ}:
\begin{equation}\label{eq:3.1}
\begin{array}{l}
  L(\shf,\shf)\fusion L(\shf,\shf)=L(\shf,0),
  \q
  L(\shf,\shf)\fusion L(\shf,\sfr{1}{16})=L(\shf,\sfr{1}{16}),
  \vsb\\
  L(\shf,\sfr{1}{16})\fusion L(\shf,\sfr{1}{16})=L(\shf,0)\oplus L(\shf,\shf).
\end{array}
\end{equation}
\end{remark}

\begin{definition}\label{df:3.2}
  (\cite{DGH})
A simple VOA $V$ is said to be \emph{framed} if there exists a set $\{e^1,
\dots,e^n\}$ of mutually orthogonal Ising vectors of $V$ such that their sum
$e^1+\cds +e^n$ is equal to the conformal element of $V$.   The subVOA  $T_n$
generated by $e^1,\dots,e^n$ is thus isomorphic to
  $L(\shf,0)^{\otimes n}$ and  is called
  a \emph{Virasoro frame} or simply  a \emph{frame} of $V$.
\end{definition}

\subsection{Structure codes}

Given a framed VOA $V$ with a frame $T_n$, one can associate two binary codes
$C$ and $D$ of length $n$ to $V$ and $T_{n}$ as follows:

Since $T_n= L(\shf,0)^{\otimes n}$ is rational, $V$ is a completely reducible
$T_n$-module. That is,
\begin{equation*}\label{2.1}
V \cong \bigoplus_{h_i\in\{0,\frac{1}{2},\frac{1}{16}\}}
m_{h_1,\ldots, h_{n}}L(\shf,h_1)\otimes\cdots\otimes L(\shf,h_n),
\end{equation*}
where the nonnegative integer $m_{h_1,\ldots,h_n}$ is the multiplicity of
$L(\shf,h_1)\otimes\cdots\otimes L(\shf, h_n)$ in $V$. Then all the
multiplicities are finite. It was also shown in \cite{DMZ} that
$m_{h_1,\ldots,h_n}$ is at most $1$ if all $h_i$ are different from $\sfr{1}{16}$.

\begin{definition}
Let $U\cong L(\shf,h_1)\otimes\cdots\otimes L(\shf,h_n) $ be an irreducible
module for $T_{n}$. We define the $\sfr{1}{16}$-word (or $\tau$-word)
$\tau(U)$ of $U$ as the binary word $\beta=(\beta_1, \dots,\beta_n)\in
\Z_2^n$ such that
\begin{equation}
\beta_i=
\begin{cases}
0 & \text{ if } h_i=0 \text{ or } \shf,\\
1 & \text{ if } h_i=\sfr{1}{16}.
\end{cases}
\end{equation}
\end{definition}

For any $\beta\in \Z_2^n$, denote by $V^\beta$ the sum of all irreducible
submodules $U$ of $V$ such that $\tau(U)=\beta$.

\begin{definition}
Define $ D:=\{\beta\in \Z_2^n \mid V^\beta \ne 0\}$. Then $D$ becomes a
binary code of length $n$ and $V$ can be written as a sum
$$V=\bigoplus_{\beta\in D} V^\beta.$$
\end{definition}

\medskip

For any $c=(c_1,\ldots,c_n)\in \Z_2^n$, denote
$M_c=m_{h_1,\ldots,h_n}L(\shf,h_1)\otimes\cdots\otimes L(\shf,h_n) $  where
$h_i=\shf$ if $c_i=1$ and $h_i=0$ elsewhere.  Note that $m_{h_1,\ldots,h_n}\leq
1$ since $h_i\neq \sfr{1}{16}$.

\begin{definition}
Define $C:=\{c\in \Z_2^n \mid M_c \neq 0\}$. Then $C$ also forms a binary code
and $V^0=\bigoplus_{c\in C}M_c$.
\end{definition}

\begin{remark}
The VOA $V^0$  is often called the {\it code VOA} associated to $C$  and is denoted
by $M_C$ (\cite{M1}).
\end{remark}
Summarizing, there exists a pair of binary codes $(C,D)$ such that
\[
V=\bigoplus_{\be\in D} V^\be \quad \text{ and } \quad V^0= \bigoplus_{c\in C}M_c.
\]
The codes $(C,D)$ are called the {\it structure codes} of a framed VOA $V$
associated to the frame $T_{n}$.
We also call the code $D$ the $\frac{1}{16}$-code and the code $C$ the
$\frac{1}2$-code of $V$ with respect to $T_n$. Note also that all $V^\be$,
$\be\in D$ , are irreducible $V^0$-modules.


Since $V$ is a VOA, its weights are integers and we have the lemma.
\begin{lemma}
\begin{enumerate}[{\rm (1)}]
\item The code $D$ is triply even, i.e., $\wt (\alpha)\equiv 0\mod 8$ for all $\alpha \in D$.

\item The code $C$ is even.
\end{enumerate}
\end{lemma}

The following theorem is also well-known (cf.\ \cite[Theorem 2.9]{DGH} and
\cite[Theorem 6.1]{Mi04}):
\begin{theorem}
Let $V$ be a framed VOA with structure codes $(C,D)$. Then, $V$ is holomorphic
if and only if $C=D^\perp$.
\end{theorem}

In \cite{LY}, the structure of a general framed VOA has been studied in detail.
In particular, the following is established (see \cite[Theorem 10]{LY}).

\begin{theorem}\label{tecode}
Let $D$ be a linear binary code of length $16k, k\in \Z_{>0}$.  Then $D$ can be
realized as the $\frac{1}{16}$-code of a holomorphic framed VOA of central
charge $8k$ if and only if (1) $D$ is triply even and (2)  the all-one vector
$(1^{16k})\in D$.
\end{theorem}

By the theorem above, the classification of the $\frac{1}{16}$-codes for
holomorphic framed VOAs is equivalent to the classification of triply even codes
of length $16k$.

\subsection{Triply  even codes of length $48$} Triply even binary codes of
length $48$ have been classified recently by Betsumiya and Munemasa \cite{BM} .
In this subsection, we will recall their result.

\begin{definition}
Let $n$ be a positive integer. We define two linear maps $d:\Z_2^n \to
\Z_2^{2n}$, $\ell :\Z_2^n \to \Z_2^{2n}$ such that
\begin{equation}\label{mapdlr}
\begin{split}
d(a_1,a_2, \dots, a_n)& = (a_1,a_1, a_2,a_2, \dots, a_n,a_n),\\
\ell(a_1,a_2, \dots, a_n)& =(a_1,0, a_2,0, \dots, a_n,0),
\end{split}
\end{equation}
for any $(a_1,a_2, \dots, a_n)\in \Z_2^n$.
\end{definition}

\begin{definition}\label{double}
 Let $E$ be a binary code of length $n$. We will  define
$$\EuD(E) = \Span_{\Z_2} \{ d(E), \ell(1^n) \}$$ to be the code
generated by $d(E)$ and $\ell(1^n)$. We call the binary code $\EuD(E)$ of length $2n$ the
\textit{extended doubling} (or simply the \textit{doubling}) of $E$.
\end{definition}

\begin{lemma}
If $E$ is a $k$-dimensional doubly even binary code of length $8n$, then the doubling $\EuD(E)$ is a
$(k+1)$-dimensional triply even code of length $16n$.
\end{lemma}

\begin{notation}
For any positive integer $n$, let $\mathcal{E}_n$ be the subcode of $\Z_2^n$
consisting of all even codewords. Note that $\EuD(E)^\perp$ contains $d(\mathcal{E}_n)$
for any binary code $E$.
\end{notation}

\begin{notation}
Let $k\geq 2$. We denote by $d_{2k}$  the doubly even binary code of length $2k$
generated by
\[
\begin{pmatrix}
1&1&1&1 & &  & & \\
&& 1&1 &1 &1 & & \\
&\quad &\quad &\quad&\quad &\ddots &\ddots &\quad&\quad   \\
&&&& & &1 &1 &1&1
\end{pmatrix}.
\]
We also denote the doubly even binary codes generated by
\[
\begin{pmatrix}
1&1&1&1 &0 &0  &0 \\
1&1 &0&0 &1 &1 &0  \\
1&0 &1 &0 &1&0 &1
\end{pmatrix}
\quad\text{ and }\quad
\begin{pmatrix}
1&1&1&1 & 1&1&1&1 \\
1&1&1&1 &0 &0  &0 &0 \\
1&1 &0&0 &1 &1 &0&0  \\
1&0 &1 &0 &1&0 &1&0
\end{pmatrix}
\]
by $e_7$ and $e_8$, respectively.
\end{notation}

\begin{notation}
Let $a_1, \dots, a_k$ be doubly even binary codes generated by its weight $4$
elements. We will  use $(a_1\cdots a_k)^+$ to denote the doubly even self-dual
code whose weight 4 elements generate a subcode $a_1\oplus \cdots \oplus a_k$.
We also use $g_{24}$ to denote the binary Golay code of length 24. Note that
$g_{24}$ has no element of weight $4$.
\end{notation}

Recently, Betsumiya and Munemasa \cite{BM} have classified all maximal triply
even binary codes of length $48$. Their main result is as follows.

\begin{theorem}[cf. \cite{BM}]
Let $D$ be a triply even code of length $48$. Then $D$ is isomorphic to a
subcode of one of the following codes:
\begin{enumerate}
         \item $\EuD(g_{24})$,   $\EuD((d_{10}e_7^2)^+)$,$\EuD(d_{24}^+)$,
            $\EuD((d_{12}^2)^+)$, $\EuD((d_{4}^6)^+)$, $\EuD((d_6^4)^+)$,
            $\EuD((d_8^3)^+) $,
         \item $\EuD(e_8)^{\oplus3}$,
         \item $\EuD(e_8)\oplus \EuD(d_{16}^+)$,
         \item $D^{\rm ex}$, a maximal triply even code of dimension $9$.
         \end{enumerate}
\end{theorem}

\begin{remark}
Note that there are exactly nine non-equivalent doubly even self-dual codes of
length $24$. The codes $g_{24}$,   $(d_{10}e_7^2)^+$, $d_{24}^+$,
$(d_{12}^2)^+$, $(d_{4}^6)^+$, $(d_6^4)^+$ and $(d_8^3)^+$ are indecomposable
while $e_8^3$ and $e_8\oplus d_{16}^+$ are decomposable.
\end{remark}

By the theorem above, we know that most triply even codes of length $48$ are
contained in some extended doublings.  The following theorem characterizes all
holomorphic framed VOAs with the $1/16$-code $D$ contained in an extended
doubling.

\begin{theorem}\cite[Theorem 3.9]{Lam}\label{VN}
Let $V=\oplus_{\be \in D }V^\be$ be a holomorphic framed VOA with the $\frac{1}{16}$-code $D$. Suppose that $D$ can be embedded into a doubling
$\EuD(E)$ for some doubly even code $E$. Then there is an even unimodular
lattice $N$ such that $V\cong V_N$ or $\tilde{V}_N$.
\end{theorem}

Because of Theorem \ref{VN}, we only concentrate on triply even codes which
are not contained in any doubling. Such codes must be subcodes of
$\EuD(e_8)^{\oplus3}$, $\EuD(e_8)\oplus \EuD(d_{16}^+)$ or $D^{ex}$.  In \cite{Lam},
holomorphic VOAs associated to subcodes of $D^{ex}$ has been studied and the
following theorem is proved.

\begin{theorem}\cite[Theorem 6.78 and Table 1]{Lam} \label{Dex}
Let $V=\oplus_{\be \in D }V^\be$ be a holomorphic framed VOA of central charge
$24$ with the $\frac{1}{16}$-code $D$.
Assume that $D$ is a subcode of $D^{\rm ex}$. Then either

\begin{enumerate}[{\rm (1)}]
\item  $V$ is isomorphic to $V\cong V_N$ or $\tilde{V}_N$ for some even unimodular
lattice $N$; or

\item  the weight one subspace $V_1$ is isomorphic to one of the Lie algebra
listed in Table \ref{LieDex}.
\end{enumerate}
Moreover, for each Lie algebra $\mathfrak{L}$ in Table \ref{LieDex}, there is a
holomorphic framed VOA $U$ such that $U_1\cong \mathfrak{L}$.
\end{theorem}

\begin{table}[bht]
\caption{ \bf Lie algebras associated to $D^{ex}$} \label{LieDex}
\begin{center}
\begin{tabular}{|c|c|c|c|c|}
\hline
No. in \cite{Sc93}& Dimension & Lie algebra \cr
\hline \hline
$7$  & $48$  & $({A_{3,4}})^3A_{1,2} $\cr
 \hline
$10$ &  $48$ & $D_{5,8}A_{1,2}$ \cr   \hline
 $18$  &  $72$  & $ A_{7,4}({A_{1,1}})^3$\cr
 \hline
$19$ &  $72$  & ${ D_{5,4}C_{3,2} (A_{1,1}})^2$\cr
\hline
$26$ &  $96$  &
  $ ({A_{5,2}})^2 C_{2,1}({A_{2,1}})^2$\cr
\hline
$33$ &  $120$ & $  A_{7,2} ({C_{3,1}})^2A_{3,1}$\cr \hline
$35$ &  $120$  & $C_{7,2}A_{3,1}$\cr\hline
$40$ &  $144$ & $  A_{9,2}A_{4,1}B_{3,1}$\cr \hline
$48$ & $192$ & $ ({C_{6,1}})^2B_{4,1}$\cr \hline
$56$ &  $288$ & $ C_{10,1}B_{6,1}$\cr \hline
\end{tabular}
\end{center}
\end{table}

Because of Theorems \ref{VN} and \ref{Dex}, we will only study holomorphic
framed VOAs associated to subcodes of $\EuD(e_8)^{\oplus3}$ and $\EuD(e_8)\oplus
\EuD(d_{16}^+)$ in the remaining of this article.

First, we note that $\EuD(e_8)\cong {\rm RM}(1,4)$ and ${\rm RM}(1,4)^\perp ={\rm RM}(2,4)$,
where ${\rm RM}(k, r)$ denotes the $k$-th order Reed-Muller code of length $2^r$. In
addition, the binary code VOA $M_{{\rm RM}(2,4)}$ is isomorphic to
$V_{\sqrt{2}E_8}^+$ (cf. \cite{Mi04}). Similarly, $\EuD(d_{16}^+)^\perp=\mathrm{span}_{\Z_2} \{
d(\mathcal{E}_{16}), \ell(d_{16}^+)\}$ and the corresponding binary code VOA
is isomorphic to $V_{\sqrt{2}D_{16}^+}^+$.
Therefore, holomorphic framed VOAs associated to subcodes of $\EuD(e_8)^{\oplus3}$ and
$\EuD(e_8)\oplus \EuD(d_{16}^+)$ are holomorphic extensions of
$(V_{\sqrt{2}E_8}^+)^{\otimes 3}$ and $V_{\sqrt{2}E_8}^+ \otimes
V_{\sqrt{2}D_{16}^+}^+$, respectively.

\section{Lattice type VOA $V_L^+$ for a totally even lattice $L$}\label{VL+}

Let $L$ be an even lattice of rank $n\in8\Z$. Assume that $L$ is totally even,
that is, $\sqrt2L^*$ is even, and that $L$ has an orthogonal basis of norm $4$.
Then $L^*\subset L/2$. Let $2^m$ be the size of $L^*/L$.
Since $L$ has an orthogonal basis of norm $4$,
$V_{L}^+$ is a framed VOA \cite{DMZ}. By \cite{DGH}, $V_{L}^+$ is rational and
$C_2$-cofinite. In this section, we review the properties of $V=V_{L}^+$ and
the set $R(V)$ of all isomorphism classes of irreducible $V$-modules.

By \cite{AD}, any irreducible $V$-module is isomorphic to one of the following:
$$\{V_{\lambda+L}^\pm,V_{L}^{T_{\chi_\lambda},\pm}\mid \lambda\in L^*/L\}.$$
Hence $|R(V)|=2^{m+2}$. We refer to \cite{Sh2} for the definition of
$\chi_\lambda$. We also use the following notations to denote the isomorphism
classes:
$$ [\lambda]^\pm=[V_{\lambda+L}^\pm],\quad  [\chi_\lambda]^\pm=[V_{L}^{T_{\chi_\lambda},\pm}].$$
By \cite{ADL}, the fusion rules of $R(V)$ are given as follows.

\begin{proposition}\label{fusion}{\rm (cf.\ \cite{ADL})} Let $L$ be an even lattice such that $\sqrt2L^*$ is even.
Then the fusion rules of $V_{L}^+$ are described as follows:
\begin{eqnarray*}
[\lambda_1]^\delta\boxtimes[\lambda_2]^\varepsilon&=&[\lambda_1+\lambda_2]^{\delta\varepsilon},\\
{} [\lambda_1]^\delta\boxtimes[\chi_{\lambda_2}]^\varepsilon&=&[\chi_{\lambda_1+\lambda_2}]^{\delta\varepsilon\nu(\lambda_2)\nu(\lambda_1+\lambda_2)},\\
{}[\chi_{\lambda_1} ]^\delta\boxtimes[\chi_{\lambda_2}]^\varepsilon&=&[\lambda_1+\lambda_2]^{\delta\varepsilon\nu(\lambda_1)\nu(\lambda_2)},
\end{eqnarray*}
where $\nu(\lambda)=+$ and $-$ if $\langle\lambda,\lambda\rangle\in2\Z$
and $1+2\Z$, respectively, and $\delta,\varepsilon\in\{\pm\}\cong\Z_2$.
\end{proposition}

By the proposition above, any irreducible $V$-module is a self-dual simple
current module. By \cite{ADL}, the associativity of $\boxtimes$ holds for $V_L^+$.
Hence $R(V)$ has an elementary abelian $2$-group structure of order $2^{m+2}$
under the fusion rules. We view $R(V)$ as an $(m+2)$-dimensional vector space
over $\F_2$.

We now assume that $n\in8\Z$. Then any irreducible $V$-module is graded by
$\Z$ or $\Z+1/2$. Let $q_V$ be the map from $R(V)$ to $\F_2$ defined by
$q_V([M])=0$ and $1$ if the weights of $M$ belong to $\Z$ and $1/2+\Z$, respectively.
Let $\langle \, ,\, \rangle$ be the $2$-form on $R(V)$ defined by $\langle
W,W'\rangle=q_V(W\boxtimes W')+q_V(W)+q_V(W')$. Then by \cite{Sh2}, it is a
symplectic form, and hence $q_V$ is a quadratic form. Moreover the type of
$q_V$ is equal to that of the quadratic form $q_L$ on $L^*/L$ defined by
$q_L(v)=\langle v,v\rangle$. By Proposition \ref{fusion}, we obtain the
following lemma directly.
\begin{lemma}\label{Inner} Let $\lambda,\mu\in L^*/L$ and $\varepsilon,\delta\in\{\pm\}$. Then the following hold:
\begin{enumerate}[{\rm (1)}]
\item $\langle[\lambda]^\varepsilon,[\mu]^\delta\rangle=2\langle\lambda,\mu\rangle$;
\item $\langle[\lambda]^+,[\chi_0]^+\rangle=0$;
\item $\langle[\lambda]^-,[\chi_0]^+\rangle=1$;
\item $\langle[0]^-,[\chi_\lambda]^\varepsilon\rangle=1$.
\end{enumerate}
\end{lemma}

\subsection{VOA $V_{\sqrt2E_8}^+$}\label{sec:3.1}
Let $E_8$ denote the $E_8$ root lattice and set $\sqrt2E_8=\{\sqrt2v\mid v\in
E_8\}$. Then $\sqrt2E_8$ is totally even, and contains an orthogonal basis of
norm $4$.
In this subsection, we review the properties of $V=V_{\sqrt2E_8}^+$.

By the previous section, $(R(V),q_V)$ is a $10$-dimensional quadratic space of
plus type over $\F_2$.
By the definition of $q_V$, $\Aut(V)$ preserves it. Hence we obtain a group
homomorphism from $\Aut(V)$ to the orthogonal group $O(R(V),q_{V})$. In fact,
it is an isomorphism by \cite[Theorem 4.5]{Sh2} (cf.\ \cite{Gr1}).

Let $M$ be an irreducible module of $V=V_{\sqrt{2}E_8}^+$. Then the lowest
weight of $M$ is $0, 1/2$ or $1$, and it is $0$ if and only if
$M\cong V_{\sqrt{2}E_8}^+$ (cf. \cite{AD,Mi04}).

Since $V_{\sqrt2E_8}=V\oplus V_{\sqrt2E_8}^-$ is a VOA, the invariant bilinear
form on the irreducible $V$-module $V_{\sqrt2E_8}^-$ is symmetric
(\cite[Proposition 5.3.6]{FHL}). By Lemma \ref{LO} (1), $\Aut(V)\cong
O(R(V),q_V)$ is transitive on the set of all isomorphism classes of $\Z$-graded
irreducible $V$-modules except for $[0]^+$. Hence the invariant bilinear form on
arbitrary $\Z$-graded irreducible $V$-module is also symmetric. Since the
dimension of the lowest weight space of any $(\Z+1/2)$-graded irreducible
$V$-module is one-dimensional, the invariant form on it must be symmetric.
\begin{lemma}\label{POVE} Let $V=V_{\sqrt2E_8}^+$.
Then the following hold:
\begin{enumerate}
\item $V$ is simple, rational, $C_2$-cofinite, self-dual and of CFT type;
\item $(R(V),q_V)$ is a non-singular $10$-dimensional quadratic space of plus type over $\F_2$;
\item $\Aut(V)\cong O(R(V),q_V)$;
\item The invariant bilinear form on arbitrary irreducible $V$-module is symmetric;
\item For $[M]\in R(V)$, the lowest weight of $M$ is $0,1/2$ or $1$, and the dimension of the lowest weight space is $1,1$ or $8$, respectively.
\end{enumerate}
\end{lemma}

The following lemma will be used in the later sections.

\begin{lemma}\label{LSVk}{\rm (cf.\ \cite{Sh6})} The automorphism group $\Aut(V^k)$ of $V^k$ is isomorphic to $\Aut(V)\wr\Sym_k$.
\end{lemma}

\subsection{VOA $V_{\sqrt2D_{16}^+}^+$}
Let $D_{16}^+$ be an even unimodular lattice of rank $16$ whose root lattice is $D_{16}$. Note that
$D_{16}^+$ is a unique indecomposable even unimodular lattice of rank $16$
up to isomorphism. Set $\sqrt2D_{16}^+=\{\sqrt2v\mid v\in D_{16}^+\}$. Then
$\sqrt2D_{16}^+$ is totally even, and it contains an orthogonal basis of norm $4$.
In this subsection, we review the properties of $X=V_{\sqrt2D_{16}^+}^+$.

Let us recall a description of $\sqrt2D_{16}^+$. Let $\{\alpha_i\mid 1\le i\le
16\}$ be an orthogonal basis of $\R^{16}$ of norm $2$ and let
$\alpha_c=\sum_{i=1}^{16}c_i\alpha_i$ for  $c=(c_i)\in\F_2^{16}$. Then
$$\sqrt2D_{16}^+\cong \sum_{1\le i,j\le
16}\Z(\alpha_i+\alpha_j)+\Z\frac{\alpha_{(1^{16})}}{2}.$$
 It is easy to see
that
\begin{equation}(\sqrt2D_{16}^+)^*\cong \sum_{1\le i\le 16}\Z\alpha_i+\sum_{c\in\mathcal{E}_{16}}\Z\frac{\alpha_c}{2}+\Z\frac{\alpha_{(1^{16})}}{4},\label{Eq:dual16}\end{equation} where $\mathcal{E}_{16}$ is the binary code of length $16$ consisting of all codewords with even weight.

\begin{lemma}\label{wt8} Let $c\in\F_2^{16}$ with $\wt(c)=8$.
Then $\sqrt2D_{16}^++\Z\alpha_c/2$ contains a sublattice isometric to
$\sqrt2E_8^{\oplus 2}$.
\end{lemma}
\begin{proof} This is clear from $$\sqrt2E_8\cong \sum_{1\le i,j\le 8}\Z(\alpha_i+\alpha_j)+\Z\frac{\alpha_{(1^{8})}}{2}.$$
\end{proof}

By the method for calculating the automorphism group of $V_L^+$ given in
\cite{Sh2}, one can show that ${\rm Aut}(X)\cong 2^{15}.(2^{14}.{\rm
Sym}_{16}).{\rm Sym}_3$.  The orbits  can then be computed directly. By the
explicit calculation, one can also show that the invariant form among
$[\lambda]^\pm$ is symmetric (cf.\ \cite{FLM}). Hence the invariant form on any
irreducible $V$-modules is symmetric since each orbit in Table \ref{Orbit}
contains an element of the form $[\lambda]^\pm$.

The fusion rules of $X$ and the quadratic form $q_X$ on $R(X)$ was described in
Section \ref{VL+} (See Proposition \ref{fusion} and Lemma \ref{Inner}).  Hence
we obtain the following lemma.

\begin{lemma}\label{POVE2} Let $X=V_{\sqrt2D_{16}^+}^+$.
Then the following hold:
\begin{enumerate}
\item $X$ is simple, rational, $C_2$-cofinite, self-dual and of CFT type;
\item $(R(X),q_X)$ is a non-singular $18$-dimensional quadratic space of plus type over $\F_2$;
\item The invariant bilinear form on arbitrary irreducible $X$-module is symmetric;
\item The orbits in $R(X)$ for the action of $\Aut(X)$ are given by Table \ref{Orbit}.
\end{enumerate}
\end{lemma}

\begin{table}[bht]
\caption{Orbits of irreducible modules of $V_{\sqrt2D_{16}^+}^+$ under ${\rm
Aut}(V_{\sqrt2D_{16}^+}^+)$.}\label{Orbit}
$$\begin{array}{|c|c|c|c|}
\hline
\hbox{orbits} & \hbox{orbit size} &  \hbox{lowest weight }&\hbox{dim. of lowest space} \\ \hline
[0]^+ & 1 & 0&1  \\ \hline
[0]^-,\, [\alpha_1]^\pm,   & 3 & 1&16 \\ \hline
[\alpha_c/2]^\pm,[\alpha_c/2-\alpha_1]^\pm\ \wt(c)=2,           & 2^2\times 120& 1/2&1 \\ \hline
[\alpha_c/2]^\pm,[\alpha_c/2-\alpha_1]^\pm\ \wt(c)=4,  & 2^2\times 1820 &1&4\\ \hline
[\alpha_c/2]^\pm,[\alpha_c/2-\alpha_1]^\pm\ \wt(c)=6,       &  2^2\times 8008 & 3/2&16 \\ \hline
[\alpha_c/2]^\pm,[\alpha_c/2-\alpha_1]^\pm\ \wt(c)=8  & 2\times 12870  &2&128\\ \hline
[\alpha_{(1^{16})}/4-\alpha_c/2]^\pm, (c\in\mathcal{E}_{16}),\ [\chi_\lambda]^+& 2^{15}+ 2^{16} & 1&1 \\ \hline
[\alpha_{(-31^{15})}/4-\alpha_c/2]^\pm,(c\in\mathcal{E}_{16}),\  [\chi_\lambda]^-& 2^{15}+ 2^{16} & 3/2&16 \\ \hline
\end{array}$$
\end{table}

\section{Framed VOAs associated to subcodes of $\EuD(e_8)^{\oplus3}$}
In this section, we will study the framed VOAs of central charge $24$ associated to subcodes of
$\EuD(e_8)^{\oplus3}$. Recall that $\EuD(e_8)\cong {\rm RM}(1,4)$ and ${\rm RM}(1,4)^\perp
={\rm RM}(2,4)$ and the binary code VOA $M_{{\rm RM}(2,4)}$ is isomorphic to
$V_{\sqrt{2}E_8}^+$. Therefore, if $U$ is a holomorphic framed VOA whose $1/16$
code is contained in $\EuD(e_8)^{\oplus3}$, then $U$ contains
$(V_{\sqrt{2}E_8}^+)^{\otimes 3}$ as a full subVOA and $U$ is a holomorphic simple current
extension of $(V_{\sqrt{2}E_8}^+)^{\otimes 3}$.

\subsection{Simple current extensions of $(V_{\sqrt2E_8}^+)^{\otimes k}$}

Let $V=V_{\sqrt2E_8}^+$. For the detail of $V$, see Section \ref{sec:3.1}. In this
subsection, we study holomorphic VOAs associated to maximal totally singular
subspaces of $(R(V)^k,q_V^k)$, which correspond to  holomorphic simple current
extensions of $V^k$.

We will  recall the relation between simple current extensions of
$(V_{\sqrt2E_8}^+)^{\otimes k}$ and totally singular subspaces of $R(V)^k$ from
\cite{Sh6}.

Let $k$ be a positive integer. We identify $R(V^k)$ with $R(V)^k$ by Lemma
\ref{LemFHL}. By Lemma \ref{POVE} (2), $(R(V)^k,q_V^k)$ is a non-singular
$10k$-dimensional quadratic space of plus type over $\F_2$.
\begin{notation}
 Let $\mathcal{T}$ be a subset of $R(V)^k$. We define $\mathfrak{V}(\mathcal{T})=\oplus_{[M]\in\mathcal{T}}M$.
\end{notation}

If $\mathcal{T}$ is a totally singular subspace, then
$\mathfrak{V}(\mathcal{T})=\oplus_{[M]\in\mathcal{T}}M$ is a VOA, which is a
simple current extension of $V^k$.  Conversely, let $U$ be a simple current
extension of $V^k$. Then
$U\cong\mathfrak{V}(\mathcal{T})=\oplus_{[M]\in\mathcal{T}}M$ for some
totally singular subspace $\mathcal{T}$ of $R(V^k)$.

\begin{proposition}\label{PS1}{\rm \cite[Proposition 4.4]{Sh6}} Let $V=V_{\sqrt2E_8}^+$.
Then the $V^k$-module $\mathfrak{V}(\mathcal{T})=\oplus_{[M]\in\mathcal{T}}M$
has a simple VOA structure of central charge $8k$ by extending its $V^k$-module
structure if and only if $\mathcal{T}$ is a totally singular subspace of
$R(V)^k$. Moreover, $\mathfrak{V}(\mathcal{T})$ is holomorphic if and only if
$\mathcal{T}$ is maximal.
\end{proposition}

Since $V$ is framed, so is $\mathfrak{V}(\mathcal{T})$. By Lemmas \ref{POVE} (3)
and \ref{LSVk},  we have $\Aut(V^k)\cong O(R(V),q_V)\wr\Sym_k$. By Lemma
\ref{LSY}, we obtain the following lemma.

\begin{lemma}\label{Lconj} Let $\mathcal{T}$ be a totally singular subspace of $R(V)^k$ and let $g$ be an automorphism of $V^k$.
Then the VOA $\mathfrak{V}(g\circ\mathcal{T})$ is isomorphic to
$\mathfrak{V}(\mathcal{T})$.
\end{lemma}

\begin{lemma}\label{PS2} Let $\mathcal{S}$ be a maximal totally singular subspace of $R(V)^k$.
\begin{enumerate}
\item If $\mathcal{S}$ contains $(a_1,0,\dots,0), (0,a_2,0,\dots,0),\dots,$ and $ (0,\dots,0,a_k)$ for some $a_i\in S(R(V))^\times$, $i=1, \dots,k$,
 then $\mathfrak{V}(\mathcal{S})$ is isomorphic to a lattice VOA $V_L$.
\item  If $\mathcal{S}$ contains $(a_1,a_2,0,\dots,0), (0,a_2,a_3,0,\dots,0),\dots,$ and $ (0,\dots,0,a_{k-1},a_k)$ for some $a_i\in S(R(V))^\times$,
$i=1,\dots,k$,  then $\mathfrak{V}(\mathcal{S})$ is isomorphic to $V_L$ or its
$\Z_2$-orbifold $\tilde{V}_L$.
\end{enumerate}
\end{lemma}

\begin{proof} Recall that $a_i$ conjugate to $[0]^-$ by Lemma \ref{POVE}.

If $\mathcal{S}$ satisfies the assumption of (1) then by Lemma \ref{Lconj},
$\mathfrak{V}(\mathcal{S})$ contains a full subVOA isomorphic to
$V_{\sqrt2E_8^{\oplus k}}$. Hence we may view $\mathfrak{V}(\mathcal{S})$ as
a simple current extension of $V_{\sqrt2E_8^{\oplus k}}$, and
$\mathfrak{V}(\mathcal{S})$ is isomorphic to $V_L$ for some even overlattice
$L$ of $\sqrt2E_8^{\oplus k}$.

If $\mathcal{S}$ satisfies the assumption of (2) then by Lemma \ref{Lconj},
$\mathfrak{V}(\mathcal{S})$ contains a full subVOA isomorphic to
$V_{\sqrt2E_8^{\oplus k}}^+$. Hence we may view $\mathfrak{V}(\mathcal{S})$
as a simple current extension of $V_{\sqrt2E_8^{\oplus k}}^+$, and
$\mathfrak{V}(\mathcal{S})$ is isomorphic to $V_L$ or $\tilde{V}_L$ for some
even overlattice $L$ of $\sqrt2E_8^{\oplus k}$.
\end{proof}

\subsection{Construction of maximal totally singular subspaces of $(R^3,q^3)$}

Let $(R,q)$ be a non-singular $2m$-dimensional quadratic space of plus type
over $\F_2$. Then $(R^3,q^3)$ is a non-singular $6m$-dimensional quadratic
space of plus type over $\F_2$. In this subsection, we study maximal totally
singular subspaces $\mathcal{S}$ of $(R^3,q^3)$.

Let us consider the following two conditions on maximal totally singular subspaces
$\mathcal{S}$ of $R^3$:
\begin{eqnarray}
(a,0,0), (0,b,0),(0,0,c)\in\mathcal{S}\setminus\{0\}\ \text{for some}\ a,b,c\in S(R)^\times, \label{Cond1}\\
(a,b,0),(0,b,c)\in\mathcal{S}\setminus\{0\}\ \text{for some}\ a,b,c\in S(R)^\times,\label{Cond2}
\end{eqnarray}
where $S(R)^\times$ is the set of all non-zero singular vectors in $R$.

\begin{remark}
If $R=R(V_{\sqrt{2}E_8}^+)$, then by Lemma \ref{PS2}, the VOA
$\mathfrak{V}(\mathcal{S})$ associated to a maximal totally singular subspace
$\mathcal{S}$ satisfying \eqref{Cond1} or \eqref{Cond2} is isomorphic to a
lattice VOA $V_L$ or its $\Z_2$-orbifold $\tilde{V}_L$.
\end{remark}

Now let us construct maximal totally singular subspaces satisfying neither (\ref{Cond1})
nor (\ref{Cond2}).

\begin{theorem}\label{TClassify} Let $S_1$ be a $k_1$-dimensional totally singular subspace of $R$ and let $S_2$ be a $k_2$-dimensional totally singular subspace of $S_1$.
Assume that $m-k_1-k_2$ is even.

Let $P$ be an $(m-k_1-k_2)$-dimensional non-singular subspace of $S_1^\perp$ of $\varepsilon$ type, where $\varepsilon\in\{\pm\}$.
Let $Q$ be a complementary subspace of $S_1$ in $(S_1\perp P)^\perp$.
Let $T$ be a complementary subspace of $S_2$ in $(S_2\perp P)^\perp$.
Let $U=Q^\perp$.
Then the following hold:

\begin{enumerate}
\item $T$ and $U$ are non-singular isomorphic quadratic spaces;
\item Let $\varphi$ be an isomorphism of quadratic spaces from $T$ to $U$ and let $\mathcal{S}(S_1,S_2,P,Q,T,\varphi)$ be the subspace of $R^3$ defined by
\begin{eqnarray*}
{\rm Span}_{\F_2}\biggl\{
(s_1,0,0),(0,s_2,0),(p,p,0),(q,0,q),(0,t,\varphi(t)) \big|\ s_i\in S_i,p\in P,q\in Q,t\in T\biggr\}.
\end{eqnarray*}
Then $\mathcal{S}(S_1,S_2,P,Q,T,\varphi)$ is a maximal totally singular
subspace of $R^3$;
\item $\mathcal{S}(S_1,S_2,P,Q,T,\varphi)$ depends only on $k_1,k_2$ and $\varepsilon$ up to $O(R,q)\wr\Sym_3$;
\item $\mathcal{S}(S_1,S_2,P,Q,T,\varphi)$ satisfies neither (\ref{Cond1}) nor (\ref{Cond2}).
\end{enumerate}

\end{theorem}
\begin{proof} It is easy to see that
\begin{equation}
\dim Q=m-k_1+k_2 \quad \text{ and } \quad \dim T=\dim U=m+k_1-k_2.\label{Eq:Dim}
\end{equation}
Since $S_i$ is the radical of $S_i^\perp$, both $T$ and $Q$ are non-singular, and so is $U$.
Since $P\perp Q$, $P\perp T$ and $Q\perp U$ are of $+$ type, the types of $P$, $Q$, $T$, $U$ are $\varepsilon$, and we obtain (1).

Clearly the generators of $\mathcal{S}(S_1,S_2,P,Q,T,\varphi)$ are singular and
they are perpendicular to each other. Hence
$\mathcal{S}(S_1,S_2,P,Q,T,\varphi)$ is totally singular. Since $$\dim
\mathcal{S}(S_1,S_2,P,Q,T,\varphi)=\dim S_1+\dim S_2+\dim P+\dim Q+\dim
T=3m,$$ it is maximal totally singular. Hence we obtain (2).

Consider  $\mathcal{S}(S_1',S_2',P',Q',T',\varphi')$ under the assumption that $\dim S_i'=\dim S_i=k_i$ for $i=1,2$ and the type of $P'$ is $\varepsilon$.
Up to the actions of $O(R,q)$ on the first and second coordinates, we may assume that $S_1'=S_1$ and $S_2'=S_2$.
Moreover, by the action of ${\rm Stab}_{O(R,q)}(S_i)$ on each coordinate, we may assume that $P'=P$, $Q'=Q$ and that $(p,p,0),(q,0,q)\in \mathcal{S}(S_1',S_2',P',Q',T',\varphi')$ for all $p\in P$ and $q\in Q$ (cf.\ Lemma \ref{LO} (2)).
Up to the action of ${\rm Stab}_{O(R,q)}(S_2)\cap{\rm Stab}_{O(R,q)}^{\rm pt}(Q)$ on the second coordinate, we obtain $T'=T$.
Furthermore, by the action of ${\rm Stab}^{\rm pt}_{O(R,q)}(Q)$ on the third coordinate, we may assume that $\varphi'=\varphi$, and hence we obtain (3).

(4) follows from the definition of $\mathcal{S}(S_1,S_2,P,Q,T,\varphi)$.
\end{proof}

\begin{remark} By (3), we denote $\mathcal{S}(S_1,S_2,P,Q,T,\varphi)$ by $\mathcal{S}(m,k_1,k_2,\varepsilon)$.
\end{remark}

\begin{theorem}\label{TClassify2} 
Let $S_1$ be a $k_1$-dimensional totally singular subspace of $R$ and let $S_2$ be a $k_2$-dimensional totally singular subspace of $S_1$.
Assume that $m-k_1-k_2$ is odd.

Let $P$ be an $(m-k_1-k_2-1)$-dimensional non-singular subspace of $S_1^\perp$ of plus type.
Let $Q$ be an $(m-k_1+k_2-1)$-dimensional non-singular subspace of $(S_1\perp P)^\perp$ of plus type.
Let $B$ be a complementary subspace of $S_1$ in $(S_1\perp P\perp Q)^\perp$.
Let $T$ be a complementary subspace of $S_2$ in $(P\perp S_2\perp B)^\perp$.
Let $U=(Q\perp B)^\perp$.
Then the following hold:

\begin{enumerate}
\item $B$ is a $2$-dimensional non-singular subspace of plus type;
\item $T$ and $U$ are isomorphic non-singular quadratic spaces of plus type;
\item Let $y$ be the non-singular vector in $B$ and let $z$ be a non-zero singular vector in $B$.
Let $\varphi$ be an isomorphism of quadratic spaces from $T$ to $U$ and set
\begin{eqnarray*}
&&\mathcal{S}(S_1,S_2,P,Q,B,T,z,\varphi)=\\
&&{\rm Span}_{\F_2}\biggr\{(s_1,0,0),(0,s_2,0),(p,p,0),(q,0,q),(0,t,\varphi(t)),(y,y,0),(y,0,y),(z,z,z)\\
&&\hspace{2cm}\bigg|\ s_i\in S_i, p\in P,q\in Q,t\in T\biggl\}.
\end{eqnarray*}
Then $\mathcal{S}(S_1,S_2,P,Q,B,T,z,\varphi)$ is a maximal totally singular subspace of $R^3$;
\item $\mathcal{S}(S_1,S_2,P,Q,B,T,z,\varphi)$ depends only on $k_1,k_2$ up to $O(R,q)\wr\Sym_3$;
\item $\mathcal{S}(S_1,S_2,P,Q,B,T,z,\varphi)$ satisfies neither (\ref{Cond1}) nor (\ref{Cond2}).

\end{enumerate}
\end{theorem}
\begin{proof} It is easy to see that $\dim B=(2m-2k_1)-(m-k_1-k_2-1)-(m-k_1+k_2-1)=2$.
Since $S_1$ is the radical of $S_1^\perp$, $B$ is non-singular.
Since $P$ and $Q$ are of plus type,  so is $B$.

It is easy to see that
\begin{equation}
\dim T=\dim U=m+k_1-k_2-1.\label{Eq:Dim2}
\end{equation}
Since $S_2$ is the radical of $S_2^\perp$, $T$ is non-singular.
Since the type of $P$ and $Q$ are the same, we obtain (2).

Clearly the generators of $\mathcal{S}(S_1,S_2,P,Q,B,T,z,\varphi)$ are singular
and they are perpendicular to each other. Hence
$\mathcal{S}(S_1,S_2,P,Q,B,T,z,\varphi)$ is totally singular. Since $$\dim
\mathcal{S}(S_1,S_2,P,Q,B,T,z,\varphi)=\dim S_1+\dim S_2+\dim P+\dim Q+\dim
T+3=3m,$$ it is maximal totally singular. Hence we obtain (3).

Consider $\mathcal{S}(S_1',S_2',P',Q',B',T',z',\varphi')$ under the assumption that $\dim S_i'=\dim S_i=k_i$ for $i=1,2$ and the type of $P'$ is plus.
Up to the actions of $O(R,q)$ on the first and second coordinates, we may assume that $S_1'=S_1$ and $S_2'=S_2$.
Moreover, up to the action of ${\rm Stab}_{O(R,q)}(S_i)$ on each coordinate (cf.\ Lemma \ref{LO} (2)), we may assume that $P'=P$, $Q'=Q$, $B'=B$, $z'=z$ and that $(p,p,0),(q,0,q)\in \mathcal{S}(S_1',S_2',P',Q',B',T',z',\varphi')$ for all $p\in P$ and $q\in Q$.
Up to the action of ${\rm Stab}_{O(R,q)}(S_2)\cap{\rm Stab}^{\rm pt}_{O(R,q)}(Q\perp B)$ on the second coordinate, we may assume that $T'=T$.
Furthermore, up to the action of ${\rm Stab}_{O(R,q)}^{\rm pt}(Q\perp B)$ in $O(R,q)$ on the third coordinate, we may assume that $\varphi'=\varphi$, and hence we obtain (4).

(5) follows from the definition of $\mathcal{S}(S_1,S_2,P,Q,B,T,z,\varphi)$.
\end{proof}

\begin{remark} By (4), we often denote $\mathcal{S}(S_1,S_2,P,Q,B,T,z,\varphi)$ by $\mathcal{S}(m,k_1,k_2)$.
\end{remark}

Next we count the numbers of certain vectors in
$\mathcal{S}(m,k_1,k_2,\varepsilon)$ and $\mathcal{S}(m,k_1,k_2)$.

\begin{lemma}\label{LNumber} Let $\mathcal{S}=\mathcal{S}(m,k_1,k_2,\varepsilon)$ or $\mathcal{S}(m,k_1,k_2)$.
\begin{enumerate}
\item The number of vectors in $\mathcal{S}$ of the form $\sigma(a,0,0)$, $a\in S(R)^\times$ and $\sigma\in \Sym_3$, is $2^{k_1}+2^{k_2}-2$.
\item The number of vectors in $\mathcal{S}$ of the form $\sigma(b,c,0)$, $b,c\in \overline{S(R)}$ and $\sigma\in \Sym_3$, is \begin{eqnarray*}\begin{cases}
2^{m-1}+2^{m+k_1-1}+2^{m+k_2-1}-3\times2^{(m+k_1+k_2)/2-1}\quad {\rm if}\quad \varepsilon=+,\ m-k_1-k_2\in2\Z,\\
2^{m-1}+2^{m+k_1-1}+2^{m+k_2-1}+3\times2^{(m+k_1+k_2)/2-1}\quad {\rm if}\quad \varepsilon=-,\ m-k_1-k_2\in2\Z,\\
2^{m-1}+2^{m+k_2-1}+2^{m+k_1-1}\quad \quad {\rm if}\quad m-k_1-k_2\in 1+2\Z.
\end{cases}
\end{eqnarray*}
\end{enumerate}
\end{lemma}
\begin{proof} The number of vectors in (1) is equal to the number of all nonzero vectors in $S_1$ and $S_2$, which is $2^{k_1}+2^{k_2}-2$.

Let $v=(a_1,a_2,a_3)$ be a vector in (2).
Then one of $a_i$ is zero.
Assume that $m-k_1-k_2\in2\Z$.
If $a_3=0$ then $a_1\in a_2+S_1$, $a_2\in  \overline{S(P)}+S_2$.
The number of such vectors is $2^{k_1+k_2}|\overline{S(P)}|$.
If $a_2=0$ then $a_1\in a_3+S_1$, $a_3\in \overline{S(Q)}$.
The number of such vectors is $2^{k_1}|\overline{S(Q)}|$.
If $a_1=0$ then  $a_2\in a_3+S_2$, $a_3\in \overline{S(U)}$.
The number of such vectors is $2^{k_2}|\overline{S(U)}|$.

Assume that $m-k_1-k_2\in2\Z+1$.
If $a_3=0$ then $a_1\in a_2+S_1$ and, $a_2\in \overline{S(P)}+S_2$ or $a_2\in y+S(P)+S_2$.
The number of such vectors is $2^{k_1+k_2}|P|$.
If $a_2=0$ then $a_1\in a_3+S_1$ and, $a_3\in \overline{S(Q)}$ or $a_3\in y+S(Q)$.
The number of such vectors is $2^{k_1}|Q|$.
If $a_1=0$ then  $a_2\in a_3+S_2$ and, $a_3\in \overline{S(U)}$ or $a_3\in y+{S(U)}$.
The number of such vectors is $2^{k_2}|U|$.
\end{proof}

\subsection{Classification of maximal totally singular subspaces of $(R^3,q^3)$}

Let $\rho_i$ denote the $i$-th coordinate projection $R^3\to R$,
$(a_1,a_2,a_3)\mapsto a_i$. For a subspace $\mathcal{S}$ of $R^3$ and distinct
$i,j\in \{1,2,3\}$, we denote  $\mathcal{S}^{(i)}=\{v\in \mathcal{S}\mid
\rho_i(v)=0\}$ and $\mathcal{S}^{(ij)}=\{v\in \mathcal{S}\mid
\rho_i(v)=\rho_j(v)=0\}$. The next theorem classifies all maximal totally singular
subspaces of $R^3$, up to $O(R,q)\wr\Sym_3$.

\begin{theorem}\label{TC} Let $\mathcal{S}$ be a maximal totally singular subspace of $R^3$.
Then up to  $O(R,q)\wr\Sym_3$, one of the following holds:
\begin{enumerate}
\item $\mathcal{S}$ satisfies (\ref{Cond1});
\item $\mathcal{S}$ satisfies (\ref{Cond2});
\item $\mathcal{S}$ is conjugate to $\mathcal{S}(S_1,S_2,P,Q,T,\varphi)$ defined as in Theorem \ref{TClassify};
\item $\mathcal{S}$ is conjugate to $\mathcal{S}(S_1,S_2,P,Q,B,T,z,\varphi)$ defined as in Theorem \ref{TClassify2}.
\end{enumerate}
\end{theorem}

\begin{proof} Let $\mathcal{S}$ be a maximal totally singular subspace of $R^3$ satisfying neither (\ref{Cond1}) nor (\ref{Cond2}).
Set $S_1=\rho_1(\mathcal{S}^{(23)})$, $S_2=\rho_2(\mathcal{S}^{(13)})$, $S_3=\rho_3(\mathcal{S}^{(12)})$ and $k_i=\dim S_i$.
Up to the action of $\Sym_3\subset O(R,q)\wr\Sym_3$, we may assume that $k_1\ge k_2\ge k_3$.
If $k_3\ge1$ then $\mathcal{S}$ satisfies (\ref{Cond1}), which is a contradiction.
Hence $k_3=0$, and $\mathcal{S}^{(12)}=0$.
Up to the action of $O(R,q)$ on the second coordinate, we may assume that $S_2\subset S_1$.

Let $\mathcal{P}$, $\mathcal{Q}$, $\mathcal{T}$ be complementary subspaces
of $\mathcal{S}^{(23)}\perp \mathcal{S}^{(13)}$ in $\mathcal{S}^{(3)}$,
$\mathcal{S}^{(23)}$ in $\mathcal{S}^{(2)}$, and  $\mathcal{S}^{(13)}$ in
$\mathcal{S}^{(1)}$, respectively. By the maximality of $\mathcal{S}$, we have
$\rho_i(\mathcal{S})=S_i^\perp$, and hence $\dim\rho_3(\mathcal{S})=2m$,
$\dim\rho_2(\mathcal{S})=2m-k_2$ and $\dim\rho_1(\mathcal{S})=2m-k_1$. It
is easy to see that $$\dim\mathcal{P}=m-k_1-k_2,\
\dim\mathcal{Q}=m-k_1+k_2,\ \dim\mathcal{T}=m+k_1-k_2.$$ Since
$\mathcal{S}$ is totally singular,
$\langle\rho_i(\mathcal{P}),\rho_i(\mathcal{Q})\rangle=\langle\rho_i(\mathcal{P}),\rho_i(\mathcal{T})\rangle=\langle\rho_i(\mathcal{T}),\rho_i(\mathcal{Q})\rangle=0$
for $i=1,2,3$. By the dimensions,
\begin{eqnarray}
(\rho_1(\mathcal{P})\perp S_1)^\perp&=&\rho_1(\mathcal{Q})\perp S_1, \label{Eq:PQS1}\\
(\rho_2(\mathcal{P})\perp S_2)^\perp&=&\rho_2(\mathcal{T})\perp S_2, \label{Eq:PTS2}\\
(\rho_3(\mathcal{Q}))^\perp&=&\rho_3(\mathcal{T}) \label{Eq:QT}
\end{eqnarray} in $R$.

Suppose that the dimension of the radical ${\rm Rad}(\rho_1(\mathcal{P}))$ of
$\rho_1(\mathcal{P})$ is greater than or equal to $2$. Then by Lemma \ref{LO}
(3), there exist non-zero singular vectors $a\in{\rm Rad}(\rho_1(\mathcal{P}))$
and $b\in R$ such that $(a,b,0) \in \mathcal{P}$. By
(\ref{Eq:PQS1}), we have $a\in\rho_1(\mathcal{Q})\perp S_1$. By the definition
of $\mathcal{P}$, $a\notin S_1$. This shows that $\mathcal{Q}\perp
\mathcal{S}^{(23)}$ contains $(a,0,c)$ for some non-zero singular vector $c\in
R$, which contradicts that $\mathcal{S}$ does not satisfy (\ref{Cond2}). Thus
$\dim {\rm Rad}(\rho_1(\mathcal{P}))\le 1$. Moreover, if $\dim {\rm
Rad}(\rho_1(\mathcal{P}))= 1$ then the non-zero vector must be non-singular.
By the same arguments, we have $\dim{\rm Rad}(\rho_1(\mathcal{Q}))\le 1$ and
$\dim{\rm Rad}(\rho_2(\mathcal{T}))\le 1$.

Assume that $\dim {\rm Rad}(\rho_1(\mathcal{P}))=0$. Then $m-k_1-k_2$ is
even. Let $\varepsilon$ be the type of $\rho_1(\mathcal{P})$. Up to the action of
${\rm Stab}_{O(R,q)}(S_2)$ on the second coordinate, we may assume that
$\mathcal{P}=\{(p,p,0)\mid p\in \rho_1(\mathcal{P})\}$. Moreover, up to the
action of $O(R,q)$ on the third coordinate, we may assume that
$\mathcal{Q}=\{(q,0,q)\mid q\in \rho_1(\mathcal{Q})\}$. By (\ref{Eq:QT}), we
have  $\rho_3(\mathcal{T})=(\rho_1(\mathcal{Q}))^\perp$. Let $\varphi$ be the map from
$\rho_2(\mathcal{T})$ to $\rho_3(\mathcal{T})$ defined by
$(0,t,\varphi(t))\in\mathcal{T}$. Since $\mathcal{T}$ is a subspace and
$q(t)=q(\varphi(t))$ for all $t\in \rho_2(\mathcal{T})$, $\varphi$ is an
isomorphism from $\rho_2(\mathcal{T})$ to $\rho_3(\mathcal{T})$. Thus
$\mathcal{S}$ is conjugate to
$\mathcal{S}(S_1,S_2,\rho_1(\mathcal{P}),\rho_1(\mathcal{Q}),\rho_2(\mathcal{T}),\varphi)$
under $O(R,q)\wr\Sym_3$.

Assume that $\dim{\rm Rad}(\rho_1(\mathcal{P}))=1$. Then $m-k_1-k_2$ is odd.
Let ${\rm Rad}(\rho_1(\mathcal{P}))=\F_2 y$. By the argument in the third
paragraph, $y$ must be non-singular. Let $P$ be a complementary subspace of
$\F_2 y$ in $\rho_1(\mathcal{P})$. Since $y$ is non-singular and is orthogonal to
$P$, we may assume that $P$ is of $+$ type. Up to the action of ${\rm
Stab}_{O(R,q)}(S_2)$ on the second coordinate, we may assume that
$\mathcal{P}=\{(p,p,0)\mid p\in P\perp\F_2 y\}$. By (\ref{Eq:PQS1}),
$y\in\rho_1(\mathcal{Q})\perp S_1$. Hence we may assume that $y\in{\rm
Rad}(\rho_1(\mathcal{Q}))$. Let $Q$ be a complementary subspace of $\F_2 y$
in $\rho_1(\mathcal{Q})$. We may assume that $Q$ is of $+$ type. Up to the
action of $O(R,q)$ on the third coordinate, we may also assume that
$\mathcal{Q}=\{(q,0,q)\mid q\in Q\perp\F_2 y\}$. Since $\mathcal{S}$ contains
$(y,y,0)$ and $(y,0,y)$, we have $(0,y,y)\in \mathcal{T}\perp
\mathcal{S}^{(13)}$. By (\ref{Eq:PQS1}), $y\in\rho_2(\mathcal{T})\perp S_2$.
Hence we may assume that $y\in\rho_2(\mathcal{T})$, and
$(0,y,y)\in\mathcal{T}$. Let $B$ be a complementary subspace of $S_1$ in
$(S_1\perp P\perp Q)^\perp$ such that $y\in B$. Let
$T=B^\perp\cap\rho_2(\mathcal{T})$ and
$U=B^\perp\cap\rho_3(\mathcal{T})$. Let $\varphi$ be the map from $T$ to
$U$ defined by $(0,t,\varphi(t))\in\mathcal{T}$. Since $\mathcal{T}$ is a
subspace and $q(t)=q(\varphi(t))$ for all $t\in T$, $\varphi$ is an isomorphism
from $T$ to $U$. Since both types of $P$ and $Q$ are $+$, the type of $B$ is
also $+$. Let $z$ be a non-zero singular vector in $B$. Set $\mathcal{S}'={\rm
Span}_{\F_2}\{\mathcal{S}^{(23)},\mathcal{S}^{(13)},\mathcal{P},\mathcal{Q},\mathcal{T}\}$.
Then $$(\mathcal{S}')^\perp/\mathcal{S}'=\Span_{\F_2}\{(z,z,z)+\mathcal{S}',
(y,0,0)+\mathcal{S}'\}.$$ Since $y$ is non-singular,
$(y,0,0)+\mathcal{S}'\notin\mathcal{S}/\mathcal{S}'$. Since $\mathcal{S}$ is
maximal totally singular, $(z,z,z)$ or $(z+y,z,z)\in\mathcal{S}$. Up to the action
of ${\rm Stab}_{O(R,q)}(S_1)\cap{\rm Stab}^{\rm pt}_{O(R,q)}(P\perp
Q\perp\F_2 y)$ on the first coordinate, we may assume that $(z,z,z)\in
\mathcal{S}$. Therefore $\mathcal{S}$ is conjugate to
$\mathcal{S}(S_1,S_2,P,Q,B,T,z,\varphi)$ under $O(R,q)\wr\Sym_3$.
\end{proof}

\subsection{Lie algebras of simple current extensions of $(V_{\sqrt2E_8}^+)^{\otimes k}$}
In this subsection, we study the Lie algebra structure of $\mathfrak{V}(\mathcal{T})$ for a totally singular subspace $\mathcal{T}$ of $R(V)^k$.



Let $S(R(V))$ denote the set of all singular vectors in $R(V)$. Set
$S(R(V))^\times=S(R(V))\setminus\{0\}$ and $\overline{S(R(V))}=R(V)\setminus S(R(V))$. The
following lemma is easy.

\begin{lemma}\label{Lw2} Let $[M]$ be a vector in $(R(V)^k,q_V^k)$.
Then $M_1\neq0$ if and only if $[M]=\sigma(a,0,\dots,0)$ or $[M]=\sigma(b,c,0,\dots,0)$ for some $a\in S(R(V))^\times$ and $b,c\in \overline{S(R(V))}$ and $\sigma\in \Sym_k$.
\end{lemma}


Define $d: R(V)^k  \to \{0,1,8\}\subset \Z$ by
$$d(v)=\begin{cases}8& \text{if}\ v=\sigma(a,0,\dots,0)\ {\rm for}\ {\rm some}\ a\in S(R(V))^\times,\ \sigma\in \Sym_k,\\
1& \text{if}\  v=\sigma(b,c,0,\dots,0)\ {\rm for}\ {\rm some}\ b,c\in \overline{S(R(V))},\ \sigma\in \Sym_k,\\
0& \text{otherwise}.\end{cases}$$
For a subset $\mathcal{T}$ of $R(V)^k$, we set $\displaystyle
{d}({\mathcal{T}})=\sum_{v\in\mathcal{T}}d(v)$.

\begin{lemma}\label{LVS2} Let $\mathcal{T}$ be a totally singular subspace of $R(V)^k$.
Then the dimension of the weight $1$ subspace of $\mathfrak{V}(\mathcal{T})$ is equal to $d(\mathcal{T})$.
\end{lemma}
\begin{proof} 
This follows from Lemma \ref{POVE} (5) and the definition of $d$.
\end{proof}

\begin{lemma}\label{L1} Let $M_a$, $M_b$ and $M_c$ be irreducible modules for $V^k$.
Assume that $(M_x)_1\neq0$ for $x=a,b,c$.
Let $\mathcal{Y}(\cdot,z)=\sum_{i\in\C}a_{(i)}z^{i-1}$ be a non-zero intertwining operator of type $M_a\times M_b\to M_c$.
Then for some non-zero vector $v\in (M_a)_1$, the map $v_{(0)}:(M_b)_1\to (M_c)_1$ is non-zero.
\end{lemma}
\begin{proof} By the assumption, $(M_x)_1\neq 0$ for $x\in\{a,b,c\}$ and $[M_a]\times[M_b]=[M_c]$.
Note that the fusion rules are preserved by the conjugation action of $\Aut(V^k)$
and thus, up to $\Aut(V)\wr\Sym_k$, we may assume that
$[M_a]=([0]^-,[0]^+,\dots,[0]^+)$ or
$([\lambda]^+,[\lambda]^+[0]^+,\dots,[0]^+)$ by Lemma \ref{Lw2} (2), where
$\lambda$ is a vector in $(\sqrt2E_8)^*$ with norm $1$.

By the explicit description of the intertwining operator (\cite{FLM}) for $[0]^-$ and $[\lambda]^+$, we have $v_{(0)}:(M_b)_1\to (M_c)_1$ is non-zero for some $v\in (M_a)_1$.
\end{proof}

\begin{lemma}\label{LVS} Let $\mathcal{T}$ be a totally singular subspace of $R(V)^k$ and $\mathcal{U}$ a subspace of $\mathcal{T}$.
\begin{enumerate}
\item The subspace $\mathfrak{V}(\mathcal{U})_1$ is a Lie subalgebra of $\mathfrak{V}(\mathcal{T})_1$.
\item Assume that $(M^1\boxtimes M^2)_1=0$ for all $[M^1]\in\mathcal{U}$, $[M^2]\in\mathcal{T}\setminus\mathcal{U}$ with $M^1_1\neq0$ and $M^2_1\neq0$.
Then $\mathfrak{V}(\mathcal{U})_1$ is an ideal of $\mathfrak{V}(\mathcal{T})_1$.
\item Assume that $(M^1\boxtimes M^2)_1=0$  for all $[M^1], [M^2]\in\mathcal{U}$ with $M^1_1\neq0$ and $M^2_1\neq0$.
Then $\mathfrak{V}(\mathcal{U})_1$ is an abelian subalgebra.
\end{enumerate}
\end{lemma}
\begin{proof} Since $\mathcal{U}$ is a subspace, $\mathfrak{V}(\mathcal{U})$ is a subVOA of $\mathfrak{V}(\mathcal{T})$.
Hence (1) holds.

For  (2), let $[M^1]\in\mathcal{U}$ and $[M^2]\in\mathcal{T}$ with
$M^1_1\neq0$ and $M^2_1\neq0$.
By the assumption for (2), $((M^1)_1)_{(0)}(M^2)_1\subset (M^3)_1=0$ if $[M^2]\in\mathcal{T}\setminus\mathcal{U}$, where $[M^3]=[M^1]\times [M^2]$.
Hence (2) holds.

(3) can be proved by the similar arguments in (2). Note also that
$[M]\times[M]=[V^k]$ for any $[M]\in \mathcal{T}$.
Therefore, $(M_1)_{(0)}M_1=0$  since $V^k_1=0$.
\end{proof}

\begin{lemma}\label{Lsemi} Let $\mathcal{T}$ be a totally singular subspace of $R(V)^k$ and $[M]$ an element in $\mathcal{T}$ with $M_1\neq0$.
Then for any $a\in M_1$,  $a_{(0)}$ is semisimple on
$\mathfrak{V}(\mathcal{T})_1$.
\end{lemma}
\begin{proof} If $d([M])=8$ then by Lemma \ref{LO} (1), we may assume that $[M]=([0]^-,[0]^+,\dots,[0]^+)$ up to $\Aut(V^k)$.
Since $(V_L^-)_1$ is the weight $1$ subspace of $V_L$, one can see that its action on $R(V^k)$ is semisimple by \cite{FLM}.

If $d([M])=1$ then $\dim M_1=1$.
Hence the semisimplicity follows from the simple current property of $M$ and Lemma \ref{L1}.
\end{proof}

\subsection{Lie algebras of holomorphic simple current extensions of $(V_{\sqrt2E_8}^+)^{\otimes3}$}
In this subsection, we consider the case $k=3$, and study the Lie algebra
structure of the weight $1$ subspace of $\mathfrak{V}(\mathcal{S})$ for a
maximal totally singular subspace $\mathcal{S}$ of $R(V)^3$.
Recall that maximal totally
singular subspaces of $(R(V)^3,q_V^3)$ were constructed in Theorems
\ref{TClassify} and \ref{TClassify2}, and were classified in Theorem \ref{TC} up to
$\Aut(V^3)\cong\Aut(V)\wr\Sym_3$.

Combining Lemmas \ref{LNumber} and \ref{LVS2}, we obtain the following.
\begin{proposition} Let $\mathcal{S}=\mathcal{S}(5,k_1,k_2,\varepsilon)$ or $\mathcal{S}(5,k_1,k_2)$ be the maximal totally singular subspace of $R(V)^3$ given in Theorem \ref{TClassify} or \ref{TClassify2}.
Then the dimension of the weight $1$ subspace of $\mathfrak{V}(\mathcal{S})$ is equal to
\begin{eqnarray*}d(\mathcal{S})=\begin{cases}3(2^{k_1+3}+2^{k_2+3}-2^{(3+k_1+k_2)/2})\quad {\rm if}\quad \varepsilon=+,\ k_1+k_2\in2\Z+1,\\
3(2^{k_1+3}+2^{k_2+3}+2^{(3+k_1+k_2)/2})\quad {\rm if}\quad \varepsilon=-,\ k_1+k_2\in2\Z+1,\\
3(2^{k_1+3}+2^{k_2+3}) \quad {\rm if}\quad k_1+k_2\in2\Z.
\end{cases}
\end{eqnarray*}

\end{proposition}

By Theorem \ref{TC} and the proposition above, we obtain the following corollary.
\begin{corollary}\label{Ex}
There are exactly $15$ maximal totally singular subspaces of $R(V)^3$ satisfying
neither (\ref{Cond1}) nor (\ref{Cond2}), up to $\Aut(V^3)$. Moreover,

$$\begin{array}{ccc}
d(\mathcal{S}(5,1,0,+))=60,& d(\mathcal{S}(5,1,0,-))=84,& d(\mathcal{S}(5,3,0,+))=192,\\ d(\mathcal{S}(5,3,0,-))=240,& d(\mathcal{S}(5,5,0,+))=744,& d(\mathcal{S}(5,2,1,+))=120,\\
d(\mathcal{S}(5,2,1,-))=168,& d(\mathcal{S}(5,4,1,+))=384,& d(\mathcal{S}(5,3,2,+))=240,\\
d(\mathcal{S}(5,0,0))=48,&d(\mathcal{S}(5,2,0))=120,& d(\mathcal{S}(5,4,0))=408,\\
d(\mathcal{S}(5,1,1))=96, &d(\mathcal{S}(5,3,1))=240,& d(\mathcal{S}(5,2,2))=192.
\end{array}$$
\end{corollary}

Let us study the Lie algebra structure of $\mathfrak{V}(\mathcal{S})$ for $\mathcal{S}$ given in Theorems \ref{TClassify} and \ref{TClassify2} by using their explicit descriptions.

\begin{lemma}\label{Lrank} Let $\mathcal{S}=\mathcal{S}(5,k_1,k_2,\varepsilon)$ or $\mathcal{S}(5,k_1,k_2)$.
\begin{enumerate}
\item If $k_1\ge 1$ then the rank of $\mathfrak{V}(\mathcal{S})_1$ is greater than or equal to $8$.
\item If $k_1\ge k_2\ge 1$ then the rank of $ \mathfrak{V}(\mathcal{S})_1$ is equal to $16$.
\end{enumerate}
\end{lemma}
\begin{proof} Let $a \in S_1\setminus\{0\}$.
Then $\mathfrak{V}(\{(a,0,0)\})_1$ is an $8$-dimensional abelian subalgebra of $\mathfrak{V}(\mathcal{S})_1$ since $a$ is conjugate to $[0]^-$.
Moreover, by Lemma \ref{Lsemi}, it is toral.
Hence (1) follows.

Assume that $k_1\ge k_2\ge1$.
Let $s_i\in S_i$ for $i=1,2$.
Set $H=\mathfrak{V}(\{(s_1,0,0),(0,s_2,0)\})_1$.
Then by Lemmas \ref{Lw2} and \ref{Lsemi}, $H$ is toral abelian and $\dim H=16$.
Let us show that $H$ is maximal abelian.

Let $v\in\mathfrak{V}(\mathcal{S})_1$ such that $v_{(0)}H=0$. It suffices to
show that $v\in H$. Let $v=\sum_{[M]\in\mathcal{S}}v_M$, where $v_M\in M_1$.
Take $[M]=(a,b,c)\in\mathcal{S}$ with $v_M\neq0$. Clearly, $M_1\neq 0$. By
Lemma \ref{Lw2}, $[M]=(a,0,0)$, $(0,b,0)$ for some $a,b\in S(R(V))^\times$ or
$[M]=(a,b,0)$, $(a,0,c)$, $(0,b,c)$ for some $a, b,c\in \overline{S(R(V))}$.

First, we consider the case where  $a,b\in \overline{S(R(V))}$ and $c=0$. Then
by the definitions of $\mathcal{S}(5,k_1,k_2,\varepsilon)$ and
$\mathcal{S}(5,k_1,k_2)$, $a,b\in S_1^\perp$ and hence
$a+s_1\in\overline{S(R(V))}$ . Since $\mathcal{S}$ is a subspace,
$(a+s_1,b,0)=(a,b,0)+(s_1,0,0)\in \mathcal{S}$. Let $[M']=(a+s_1,b,0)$. Then
$M'_1\neq 0$ and by Lemma \ref{L1}, there is $h\in
\mathfrak{V}(\{(s_1,0,0)\})_1\subset H$ such that $(v_M)_{(0)}h\neq0$. Since
the projection of $v_{(0)}h$ to $M'$ is $(v_M)_{(0)}h$, we have
$v_{(0)}h\neq0$, which contradicts $v_{(0)}H=0$. Hence $v_M=0$.
By the same arguments, $v_M=0$
if $[M] =(a,b,0)$, $(a,0,c)$ or $(0,b,c)$ with $a,b,c \in \overline{S(R(V))}$.

Next, we consider the case where $[M]=(a,0,0)$, $(0,b,0)$ for some $a,b\in S(R(V))^\times$.
Then $[M]$ belongs to $\{(x,y,0)\mid x\in
S_1,y\in S_2\}$. Since $v_{(0)}H=0$, it is easy to see that
$M_1\subset H$ by the similar arguments. Hence $v_M\in H$, and $v\in H$. Thus
$H$ is maximal abelian, and the rank of $\mathfrak{V}(\mathcal{S})_1$ is $16$.
\end{proof}

\begin{lemma}\label{Lrank2} Let $\mathcal{S}=\mathcal{S}(5,k_1,k_2,\varepsilon)$ or $\mathcal{S}(5,k_1,k_2)$.
\begin{enumerate}
\item If $k_1= 2,3$ and $4$ then $\mathfrak{V}(\mathcal{S})_1$ has a semi-simple Lie subalgebra of type $(A_{1,1})^8$, $(D_{4,1})^2$ and $D_{8,1}$, respectively.
\item If $k_1\ge k_2\ge 2$ then $ \mathfrak{V}(\mathcal{S})_1$ has a semi-simple Lie subalgebra of type $(A_{1,1})^{16}$.
\item If $k_1=3$ and $k_2=2$ then $ \mathfrak{V}(\mathcal{S})_1$ has a semi-simple Lie subalgebra of type $(D_{4,1})^2(A_{1,1})^{8}$.
\end{enumerate}
\end{lemma}
\begin{proof} Recall that for $2$, $3$ and $4$ -dimensional totally singular subspaces $U$ of $R(V)$, $\oplus_{[M]\in U}M$ are lattice VOAs associated to root lattices $A_1^8$, $D_4^2$ and $D_8$, respectively.
Hence we obtain this lemma by Lemma \ref{LVS} (1).
\end{proof}

\begin{lemma}\label{LIL3} Let $\mathcal{S}=\mathcal{S}(5,k_1,0,\varepsilon)$ be the maximal totally singular subspace of $R(V)^3$ given in Theorem \ref{TClassify}.
Let $\mathcal{T}=\{(0,t,\varphi(t))\mid t\in T\}$.
\begin{enumerate}
\item The subspace $\mathfrak{V}(\mathcal{T})_1$ is an ideal of $\mathfrak{V}(\mathcal{S})_1$.
\item The rank of $\mathfrak{V}(\mathcal{T})_1$ is $2$, $4$, $8$ if $\dim (T)=4,6,8$, respectively.
\end{enumerate}
\end{lemma}
\begin{proof} It is easy to see that $\mathcal{T}$ satisfies $(M\boxtimes M')_1=0$ for all $[M]\in\mathcal{T}$, $[M']\in\mathcal{S}\setminus\mathcal{T}$ with $M_1\neq0$ and $M'_1\neq0$.
Then (1) follows from Lemma \ref{LVS} (2).

Let $T_0$ be a maximal subset of $\overline{S(T)}$ such that $a+b\in S(T)$ for all distinct $a,b\in T_0$.
Set $\mathcal{T}_0=\{(0,t,\varphi(t))\mid t\in T_0\}$.
Then $\mathfrak{V}(\mathcal{T}_0)_1$ is a Cartan subalgebra of $\mathfrak{V}(\mathcal{T})_1$ by Lemmas \ref{LVS} (3) and \ref{Lsemi}.
It is easy to see that $|T_0|=2,4,8$ if $\dim T=4,6,8$.
Hence we have (2).
\end{proof}

\begin{lemma}\label{LIL5} Let $\mathcal{S}=\mathcal{S}(5,k_1,1,\varepsilon)$ be the maximal totally singular subspace of $R(V)^3$ given in Theorem \ref{TClassify}.
Let $H$ be a Cartan subalgebra of $\mathfrak{V}(\mathcal{S})_1$ contained in $\mathfrak{V}(\{(s_1,s_2,0)\mid s_i\in S_i\})_1$.
Let $\mathcal{T}=\{(0,t+s,\varphi(t))\mid t\in T,\ s\in S_2\}\setminus\{(0,s,0)\mid s\in S_2\}$.
Then $\mathfrak{V}(\mathcal{T})_1$ is a sum of root spaces corresponding to a sum of irreducible components of root systems of $\mathfrak{V}(\mathcal{S})_1$.
\end{lemma}
\begin{proof} Clearly, $H$ preserves $\mathfrak{V}(\mathcal{T})_1$.
Hence $\mathfrak{V}(\mathcal{T})_1$ is a sum of root spaces.

Let $[M]\in\mathcal{S}$ with $M_1\neq0$. Then $(M\boxtimes M')_1\neq0$
for some $[M']\in\mathcal{T}$ if and only if $[M]\in\mathcal{T}\cup
\{(0,s,0)\mid s\in S_2\}$. Hence we obtain this lemma.
\end{proof}

\begin{lemma}\label{LIL6} Let $\mathcal{S}=\mathcal{S}(5,k_1,0)$ be the maximal totally singular subspace of $R(V)^3$ given in Theorem \ref{TClassify2}.
Let $\mathcal{T}=\{(0,t,\varphi(t))\mid t\in T\perp\F_2y\}\setminus\{(0,y,y)\}$, where $\varphi(y)=y$.
\begin{enumerate}
\item The subspace $\mathfrak{V}(\mathcal{T})_1$ is an ideal of $\mathfrak{V}(\mathcal{S})_1$.
\item Let $2m$ be the dimension of $T$.
Then the rank of $\mathfrak{V}(\mathcal{T})_1$ is $2^m-1$.
\end{enumerate}
\end{lemma}
\begin{proof} It is easy to see that $\mathcal{T}$ satisfies $(M\boxtimes M')_1=0$ for all $[M]\in\mathcal{T}$, $[M']\in\mathcal{S}\setminus\mathcal{T}$ with $M_1\neq0$ and $M'_1\neq0$.
Then (1) follows from Lemma \ref{LVS} (2).

Let $U$ be a maximal totally singular subspace of $T$. Then $\dim U=m$. Set
$\mathcal{U}=\{(0,a+y,\varphi(a)+y)\mid a\in U\}\setminus\{(0,y,y)\}$ and $H=\mathfrak{V}(\mathcal{U})_1$. Then
$M_1\neq0$ for all $[M]\in\mathcal{U}$. Moreover for $[M],[M']\in\mathcal{U}$,
$(M\boxtimes M')_1=0$. By Lemma \ref{LVS} (3),
$H$ is abelian, and its dimension is $2^m-1$. By
Lemma \ref{Lsemi}, it is toral. Let us show that $H$ is maximal.

Let $v\in \mathfrak{V}(\mathcal{T})_1$ with $v_{(0)}H=0$.
Let $v=\sum_{[M]\in\mathcal{T}}v_M$, where $v_M\in M_1$.
Suppose that $v_M\neq0$ for $[M]=(0,b+y,\varphi(b)+y)$, $b\in S(T)$. If  $b\notin U$ then there
exists $c\in U$ such that $\langle b,c\rangle=1$, then $(v_M)_{(0)}M'_1\neq0$,
where $[M']=(0,c+y,\varphi(c)+y)$. It contradicts $v_{(0)}H=0$. Hence $b\in U$. Suppose that
$v_M\neq0$ for $[M]=(0,b,\varphi(b))$, $b\in \overline{S(T)}$. Then there exists $c\in U$
such that $\langle b,c\rangle=1$, then $(v_M)_{(0)}M'\neq0$, where
$[M']=(0,c+y,\varphi(c)+y)$. This contradicts the assumption. Thus $v\in
 H$, and we obtain (2).
\end{proof}

Later, we use the following lemma about holomorphic VOAs of central charge $16$.

\begin{lemma}\label{Lholo16}
\begin{enumerate}
\item The $\Z_2$-orbifold of $V_{E_8\oplus E_8}$ is isomorphic to $V_{D_{16}^+}$, and that of $V_{D_{16}^+}$ is isomorphic to $V_{E_8\oplus E_8}$.
\item Let $\mathcal{T} =\{ (a, a)\mid a\in R(V)\}$.
Then $\mathfrak{V}(\mathcal{T})\cong V_{D_{16}^+}$.
\end{enumerate}
\end{lemma}
\begin{proof} By \cite[Lemma 3.4]{Sh3} both $V_{E_8\oplus E_8}^+$ and $V_{D_{16}^+}^+$ are isomorphic to $V_N$ for some even lattice $N$ such that $E_8\oplus E_8$ and $D_{16}^+$ are even overlattices of $N$ of index $2$.
Hence we obtain (1).

Let $U_0$ be a maximal totally singular subspace of $R(V)$.
Let $\mathcal{T}_0=\{(a,b)\mid a,b\in U_0\}$.
Then $\mathfrak{V}(\mathcal{T}_0)\cong V_{E_8\oplus E_8}$.
Consider its $\Z_2$-orbifold associated to the lift of $(-1,-1)\in\Aut(\sqrt2E_8\oplus \sqrt2E_8)$.
Then it is isomorphic to $\mathfrak{V}(\mathcal{T}_1)$, where $\mathcal{T}_1=\{(a,b)\mid a,b\in U_1\}\oplus\{(a,a)\mid a\in U'_1\}$, where $U_1$ is a codimension $1$ subspace of $U$ and $U'_1$ is a complementary subspace of $U_1$ in $U_1^\perp$.
By (1), $\mathfrak{V}(\mathcal{T}_1)\cong V_{D_{16}^+}$.

Let $\mathcal{T}_i=\{(a,b)\mid a,b\in U_i\}\oplus\{(a,a)\mid a\in U'_i\}$, where $U_i$ is a codimension $1$ subspace of $U_{i-1}$ and $U'_i$ is a complementary subspace of $U_i$ in $U_i^\perp$.
By the same argument, we obtain $\mathfrak{V}(\mathcal{T}_2)\cong V_{E_8\oplus E_8}$, $\mathfrak{V}(\mathcal{T}_3)\cong V_{D_{16}^+}$, $\mathfrak{V}(\mathcal{T}_4)\cong V_{E_8\oplus E_8}$, and $\mathfrak{V}(\mathcal{T}_5)\cong V_{D_{16}^+}$.
Since $\mathcal{T}_5=\mathcal{T}$, we obtain this lemma.
\end{proof}

\subsection{Determination of the Lie algebra structure of $\mathfrak{V}(\mathcal{S})$ for $\mathcal{S}\subset R(V)^3$}
In this subsection, we determine the Lie algebra structure of the weight $1$ subspace of $\mathfrak{V}(\mathcal{S})$ for $\mathcal{S}=\mathcal{S}(5,k_1,k_2,\varepsilon)$ and $\mathcal{S}(5,k_1,k_2)$.

First we will recall several important results from \cite{DMb,DM4}.
\begin{proposition}\label{PDM1}{\rm \cite[Theorem 3 and (3.6)]{DMb}} Let $V$ be a $C_2$-cofinite holomorphic VOA of CFT type. Suppose the central charge of $V$ is $24$.
\begin{enumerate}[{\rm (a)}]
\item The Lie
algebra $V_1$ has Lie rank less than or equal to $24$ and  is either abelian
(including $0$) or semisimple.
\item Suppose $V_1$ is semisimple, that is,
\[
V_1=\mathfrak{g}_{1,k_1}\oplus \mathfrak{g}_{2,k_2}\oplus \cdots\oplus \mathfrak{g}_{n,k_n},
\]
where $\mathfrak{g}_i$ is a simple Lie algebra whose affine Lie algebra has level $k_i$ on $V$.
Then
\begin{equation}\label{rl}
\frac{h_i^\vee}{k_i}= \frac{(\dim V_1 -24)}{24},
\end{equation}
where $h_i^{\vee}$ is the dual Coxeter number of $\mathfrak{g}_i$.
In particular, the ratio ${h_i^\vee}/{k_i}$ is independent of $\mathfrak{g}_i$.
\end{enumerate}
\end{proposition}

\begin{proposition}\label{PDM2}{\rm \cite[Theorem 3.1]{DM4}}
Let $V$ be a simple self-dual VOA which is $C_2$-cofinite and of CFT
type.
Let $\mathfrak{g}$ be a simple Lie subalgebra of $V_1$, $k$ the level of $V$ as
module for the corresponding affine Lie algebra.
Then $k$ is a positive integer.
\end{proposition}

\begin{remark} Let $\mathcal{T}$ be a totally singular subspace of $R(V)^k$.
Since $\mathfrak{V}(\mathcal{T})$ is framed, it is simple, rational, $C_2$-cofinite and of CFT type (\cite{DGH}).
\end{remark}

\noindent {\bf Case: $\mathcal{S}=\mathcal{S}(5,1,0,+)$.}  By Lemma \ref{LVS2}
and Corollary \ref{Ex}, $\dim \mathfrak{V}(\mathcal{S})_1=60$.
By Proposition \ref{PDM1}, each simple component $\mathfrak{g}_i$ of
$\mathfrak{V}(\mathcal{S})_1$ satisfies $h_i/k_i=3/2$. Hence by Proposition
\ref{PDM2}, $\mathfrak{g}_i$ is one of  the following.
\[
\begin{array} {l|c|c|c|c|c|c|c}
\text{Type} & A_{2,2} &  A_{5,4}&C_{2,2}&B_{5,6}&C_{5,4}&D_{4,4}&F_{4,6}\\     \hline
\ \ h^\vee      &  3 & 6 &3 & 9 &6 & 6 & 9 \\ \hline
\text{Dimension}& 8&35&10&55&55&28& 52
\end{array}
\]
By Lemma \ref{LIL3}, $\mathfrak{V}(\mathcal{S})_1$ has a $28$-dimensional
ideal with rank $4$. Hence $\mathfrak{V}(\mathcal{S})_1$ is isomorphic to
$D_{4,4}(A_{2,2})^4$.

\begin{proposition} The Lie algebra structure of $\mathfrak{V}(\mathcal{S}(5,1,0,+))_1$ is $D_{4,4}(A_{2,2})^4$.
\end{proposition}

\noindent {\bf Case: $\mathcal{S}=\mathcal{S}(5,1,0,-)$.}
By Lemma \ref{LVS2} and Corollary \ref{Ex}, $\dim
\mathfrak{V}(\mathcal{S})_1=84$. By Proposition \ref{PDM1}, each simple
component $\mathfrak{g}_i$ of $\mathfrak{V}(\mathcal{S})_1$ satisfies
$h_i/k_i=5/2$. Hence by Proposition \ref{PDM2}, $\mathfrak{g}_i$ is one of the
following.
\[
\begin{array} {l|c|c|c|c}
\text{Type} &  A_{4,2} &B_{3,2}& C_{4,2} &D_{6,4}       \\     \hline
\ \ h^\vee      &  5 & 5 &5 & 10 \\ \hline
\text{Dimension}& 24& 21& 36&66
\end{array}
\]
By Lemma \ref{LIL3}, $\mathfrak{V}(\mathcal{S})_1$ has a $36$-dimensional
ideal with rank $4$. Hence $\mathfrak{V}(\mathcal{S})_1$ must be isomorphic
to $C_{4,2}(A_{4,2})^2$.

\begin{proposition} The Lie algebra structure of $\mathfrak{V}(\mathcal{S}(5,1,0,-))_1$ is $C_{4,2}(A_{4,2})^2$.
\end{proposition}

\noindent {\bf Case: $\mathcal{S}=\mathcal{S}(5,3,0,+)$.}
By Lemma \ref{LVS2} and Corollary \ref{Ex}, $\dim
\mathfrak{V}(\mathcal{S})_1=192$. By Proposition \ref{PDM1}, each simple
component $\mathfrak{g}_i$ of $\mathfrak{V}(\mathcal{S})_1$ satisfies
$h_i/k_i=7$. Hence by Proposition \ref{PDM2}, $\mathfrak{g}_i$ is one of the
following.
\[
\begin{array} {l|c|c|c|c}
\text{Type} & A_{6,1} &B_{4,1} &C_{6,1} &D_{8,2}      \\     \hline
\ \ h^\vee      &  7 & 7 &7 & 14 \\ \hline
\text{Dimension}& 48& 36 & 78 & 120
\end{array}
\]
By Lemma \ref{LIL3}, $\mathfrak{V}(\mathcal{S})_1$ has a $120$-dimensional
ideal with rank $8$. Hence $\mathfrak{V}(\mathcal{S})_1$ must be isomorphic
to $D_{8,2}(B_{4,1})^2$.

\begin{proposition} The Lie algebra structure of $\mathfrak{V}(\mathcal{S}(5,3,0,+))_1$ is $D_{8,2}(B_{4,1})^2$.
\end{proposition}

\noindent {\bf Case: $\mathcal{S}=\mathcal{S}(5,3,0,-)$.}
By Lemma \ref{LVS2} and Corollary \ref{Ex}, $\dim
\mathfrak{V}(\mathcal{S})_1=240$. By Proposition \ref{PDM1}, each simple
component $\mathfrak{g}_i$ of $\mathfrak{V}(\mathcal{S})_1$ satisfies
$h_i/k_i=9$. Hence by Proposition \ref{PDM2},  $\mathfrak{g}_i$ is one of the
following.
\[
\begin{array} {l|c|c|c|c|c|c}
\text{Type} & A_{8,1} &B_{5,1} & C_{8,1} &D_{10,2} & E_{7,2} &F_{4,1}   \\     \hline
\ \ h^\vee      &  9 & 9 &9 & 18 & 18 &9 \\ \hline
\text{Dimension}& 80 & 55& 136 & 190 &133 & 52
\end{array}
\]
By Lemma \ref{LIL3}, $\mathfrak{V}(\mathcal{S})_1$ has a $136$-dimensional
ideal with rank $8$. Hence $\mathfrak{V}(\mathcal{S})_1$ must be isomorphic
to $C_{8,1}(F_{4,1})^2$.

\begin{proposition} The Lie algebra structure of $\mathfrak{V}(\mathcal{S}(5,3,0,-))_1$ is $C_{8,1}(F_{4,1})^2$.
\end{proposition}

\noindent {\bf Case: $\mathcal{S}=\mathcal{S}(5,5,0,+)$.}
By Lemma \ref{LVS2} and Corollary \ref{Ex}, $\dim
\mathfrak{V}(\mathcal{S})_1=744$. By Proposition \ref{PDM1}, each simple
component $\mathfrak{g}_i$ satisfies $h_i/k_i=30$. Hence by Proposition
\ref{PDM2}, $\mathfrak{g}_i$ is  $D_{16,1}$ or $E_{8,1}$. Note that their
dimensions are $496$ and $248$. Hence $\mathfrak{V}(\mathcal{S})_1$ is
isomorphic to $(E_{8,1})^3$ or $E_{8,1}D_{16,1}$.

By the construction in Theorem \ref{TClassify}, $S_1$ is a
$5$-dimensional totally singular space. Then we have $P=Q=0$ and $T=U=R(V)$.
Hence, $\mathcal{S}(5,5,0,+)$ is spanned by
\[
\begin{split}
&(a, 0,0), \qquad a\in S_1,\\
&(0, b, b), \qquad b\in R(V).
\end{split}
\]
Since $S_1$ is a maximal totally singular subspace of $R(V)$,
$\mathfrak{V}(\{(a,0,0)\mid a\in S_1\})$ is isomorphic to $V_{E_8}\otimes V^{\otimes2}$.
By Lemma \ref{Lholo16}, $\mathfrak{V}(\{(0,b,b)\mid b\in R(V)\})$ is isomorphic to $V\otimes V_{D_{16}^+}$.
Hence we obtain the following.

\begin{proposition}
The VOA $\mathfrak{V}(\mathcal{S}(5,5,0,+))$ is isomorphic to the lattice VOA associated to the Niemeier lattice $N(E_{8}D_{16})$.
\end{proposition}

\noindent {\bf Case: $\mathcal{S}=\mathcal{S}(5,2,1,+)$.}
By Lemma \ref{LVS2} and Corollary \ref{Ex}, $\dim
\mathfrak{V}(\mathcal{S})_1=120$. By Proposition \ref{PDM1}, each simple
component $\mathfrak{g}_i$ of $\mathfrak{V}(\mathcal{S})_1$ satisfies
$h_i/k_i=4$. Hence by Proposition \ref{PDM2}, $\mathfrak{g}_i$ is one of the
following.
\[
\begin{array} {l|c|c|c|c|c|c|c|c}
\text{Type} &A_{3,1} & A_{7,2} &C_{3,1} &C_{7,2} & D_{5,2} & D_{7,3} &E_{6,3}&G_{2,1}      \\     \hline
\ \ h^\vee      &  4 & 8 &4 & 8& 8 & 12 &12 & 4 \\ \hline
\text{Dimension}& 15 &63 &21 & 105 & 45 &91 & 78 &14
\end{array}
\]
By Lemma \ref{LIL5}, $\mathfrak{V}(\mathcal{S})_1$ contains an ideal with
$56$-dimensional root space. Hence $\mathfrak{V}(\mathcal{S})_1$ is
isomorphic to $A_{7,2}(C_{3,1})^2A_{3,1}$ or  $A_{7,2}A_{3,1}(G_{2,1})^3$. By
Lemma \ref{Lrank2}, $\mathfrak{V}(\mathcal{S})_1$ contains $(A_{1,1})^8$.
Hence it must be isomorphic to $A_{7,2}(C_{3,1})^2A_{3,1}$.

\begin{proposition} The Lie algebra structure of $\mathfrak{V}(\mathcal{S}(5,2,1,+))_1$ is $A_{7,2}(C_{3,1})^2A_{3,1}$.
\end{proposition}

\noindent {\bf Case: $\mathcal{S}=\mathcal{S}(5,2,1,-)$.}
By Lemma \ref{LVS2} and Corollary \ref{Ex}, $\dim
\mathfrak{V}(\mathcal{S})_1=168$. By Proposition \ref{PDM1}, each simple
component $\mathfrak{g}_i$ of $\mathfrak{V}(\mathcal{S})_1$ satisfies
$h_i/k_i=6$. Hence by Proposition \ref{PDM2}, $\mathfrak{g}_i$ is one of the
following.
\[
\begin{array} {l|c|c|c|c|c|c|c}
\text{Type} & A_{5,1} & A_{11,2} & C_{5,1} & D_{4,1} & D_{7,2} & E_{6,2} & E_{7,3}       \\     \hline
\ \ h^\vee      &  6 & 12 &6 & 6& 12& 12 &18 \\ \hline
\text{Dimension}&      35 & 143 & 55 & 28 & 91 &78    &133
\end{array}
\]
By Lemma \ref{LIL5}, $\mathfrak{V}(\mathcal{S})_1$ contains an ideal with
$72$-dimensional root space. Hence $\mathfrak{V}(\mathcal{S})_1$ is
isomorphic to $E_{6,2}C_{5,1}A_{5,1}$,

\begin{proposition} The Lie algebra structure of $\mathfrak{V}(\mathcal{S}(5,2,1,-))_1$ is $E_{6,2}C_{5,1}A_{5,1}$.
\end{proposition}

\noindent {\bf Case: $\mathcal{S}=\mathcal{S}(5,4,1,+)$.}
By Lemma \ref{LVS2} and Corollary \ref{Ex}, $\dim
\mathfrak{V}(\mathcal{S})_1=384$. By Proposition \ref{PDM1}, each simple
component $\mathfrak{g}_i$ of $\mathfrak{V}(\mathcal{S})_1$ satisfies
$h_i/k_i=15$. Hence by Proposition \ref{PDM2}, $\mathfrak{g}_i$ is one of the
following.
\[
\begin{array} {l|c|c|c}
\text{Type} &A_{14,1}&B_{8,1} &E_{8,2}      \\     \hline
\ \ h^\vee      &  15 & 15 &30  \\ \hline
\text{Dimension}& 224 &136 &248
\end{array}
\]
Hence $\mathfrak{V}(\mathcal{S})_1$ is isomorphic to $E_{8,2}B_{8,1}$.

\begin{proposition} The Lie algebra structure of $\mathfrak{V}(\mathcal{S}(5,4,1,+))_1$ is $E_{8,2}B_{8,1}$.
\end{proposition}

\begin{remark} In \cite{Sc93}, $E_{8,2}B_{8,1}$ was written as $E_{8,2}B_{6,1}$, which is a misprint.
\end{remark}

\noindent {\bf Case: $\mathcal{S}=\mathcal{S}(5,3,2,+)$.}
By Lemma \ref{LVS2} and Corollary \ref{Ex}, $\dim
\mathfrak{V}(\mathcal{S})_1=240$. By Proposition \ref{PDM1}, each simple
component $\mathfrak{g}_i$ of $\mathfrak{V}(\mathcal{S})_1$ satisfies
$h_i/k_i=9$. Hence by Proposition \ref{PDM2}, $\mathfrak{g}_i$ is one of the
following.
\[
\begin{array} {l|c|c|c|c|c|c}
\text{Type} & A_{8,1} & B_{5,1} & C_{8,1} & D_{10,2} & E_{7,2} & F_{4,1}     \\     \hline
\ \ h^\vee      &  9 & 9 &9 & 18&18 &9 \\ \hline
\text{Dimension}& 80 & 55 & 136 & 190 & 133 & 52
\end{array}
\]
Hence $\mathfrak{V}(\mathcal{S})_1$ is isomorphic to $(A_{8,1})^3$,
$C_{8,1}(F_{4,1})^2$ or $B_{5,1}E_{7,2}F_{4,1}$. By Lemma \ref{Lrank}, the rank
of $\mathfrak{V}(\mathcal{S})_1$ is $16$. Moreover, by Lemma \ref{Lrank2},
$\mathfrak{V}(\mathcal{S})_1$ has a subalgebra $(D_{4,1})^2(A_{1,1})^{8}$.
Hence $\mathfrak{V}(\mathcal{S})_1$ must be isomorphic to
$C_{8,1}(F_{4,1})^2$.

\begin{proposition} The Lie algebra structure of $\mathfrak{V}(\mathcal{S}(5,3,2,+))_1$ is $C_{8,1}(F_{4,1})^2$.
\end{proposition}

\noindent {\bf Case: $\mathcal{S}=\mathcal{S}(5,0,0)$.}
By Lemma \ref{LVS2} and Corollary \ref{Ex}, $\dim \mathfrak{V}(\mathcal{S})_1=48$.
By Propositions \ref{PDM1}, each simple component $\mathfrak{g}_i$ of $\mathfrak{V}(\mathcal{S})_1$ satisfies $h_i/k_i=1$.

We use the notation in Theorem \ref{TClassify2}. Set $\mathcal{T}={\rm
Span}_{\F_2}\{(y,y,0),(y,0,y)\}\subset\mathcal{S}$. Then
$\mathfrak{V}(\mathcal{T})_1$ is a $3$-dimensional ideal of
$\mathfrak{V}(\mathcal{S})_1$ isomorphic to  a Lie algebra of type $A_1$.
Moreover, by Lemma \ref{LIL6},  $\mathfrak{V}(\mathcal{S})_1$ has a
$15$-dimensional ideal
$$\mathfrak{V}(\{(a+y,a+y,0),(b,b,0)\mid a\in S(P)^\times, b\in \overline{S(P)}\})_1$$
of  rank $3$.  By the same argument in the proof of Lemma \ref{LIL6},
\begin{eqnarray*}
\mathfrak{V}(\{(a+y,0,a+y),(b,0,b)\mid a\in S(Q)^\times, b\in \overline{S(Q)}\})_1,\\
\mathfrak{V}(\{(0,a+y,\varphi(a)+y),(0,b,\varphi(b))\mid a\in S(T)^\times, b\in \overline{S(T)}\})_1,
\end{eqnarray*}
are $15$-dimensional ideals of  rank $3$. Hence by Proposition \ref{PDM2},
$\mathfrak{V}(\mathcal{S})_1$ is isomorphic to $(A_{3,4})^3A_{1,2}$.

\begin{proposition}
The Lie algebra structure of $\mathfrak{V}(\mathcal{S}(5,0,0))_1$ is $(A_{3,4})^3A_{1,2}$.
\end{proposition}

\noindent {\bf Case: $\mathcal{S}=\mathcal{S}(5,2,0)$. }
By Lemma \ref{LVS2} and Corollary \ref{Ex}, $\dim
\mathfrak{V}(\mathcal{S})_1=120$. By Proposition \ref{PDM1}, each simple
component $\mathfrak{g}_i$ of $\mathfrak{V}(\mathcal{S})_1$ satisfies
$h_i/k_i=4$. Hence by Proposition \ref{PDM2}, $\mathfrak{g}_i$ is one of the
following.
\[
\begin{array} {l|c|c|c|c|c|c|c|c}
\text{Type} &A_{3,1} & A_{7,2} & C_{3,1} & C_{7,2} & D_{5,2} & D_{7,3} & E_{6,3} &  G_{2,1}     \\     \hline
\ \ h^\vee      &  4 & 8 &4 & 8 &8 & 12 &12 &4 \\ \hline
\text{Dimension}& 15 & 63 & 21 & 105 & 45 & 91 & 78  & 14
\end{array}
\]
By Lemma \ref{LIL6}, $\mathfrak{V}(\mathcal{S})_1$ has a $63$-dimensional
ideal. Hence $\mathfrak{V}(\mathcal{S})_1$ must be isomorphic to
$A_{7,2}(C_{3,1})^2A_{3,1}$ or $A_{7,2}A_{3,1}(G_{2,1})^3$. By Lemma
\ref{Lrank2}, the $57$-dimensional ideal of $\mathfrak{V}(\mathcal{S})_1$
contains $(A_{1,1})^8$. Hence $\mathfrak{V}(\mathcal{S})_1$ must be isomorphic
to $A_{7,2}(C_{3,1})^2A_{3,1}$.

\begin{proposition} The Lie algebra structure of $\mathfrak{V}(\mathcal{S}(5,2,0))_1$ is $A_{7,2}(C_{3,1})^2A_{3,1}$.
\end{proposition}

\noindent {\bf Case: $\mathcal{S}=\mathcal{S}(5,4,0)$. }
By Lemma \ref{LVS2} and Corollary \ref{Ex}, $\dim
\mathfrak{V}(\mathcal{S})_1=408$. By Proposition \ref{PDM1}, each simple
component $\mathfrak{g}_i$ of $\mathfrak{V}(\mathcal{S})_1$ satisfies
$h_i/k_i=16$. Hence by Proposition \ref{PDM2}, $\mathfrak{g}_i$ is of type
$A_{15,1}$ or  $D_{9,1}$. Note that their dimensions are $255$ and $153$.
Hence $\mathfrak{V}(\mathcal{S})_1$ is isomorphic to $A_{15,1}D_{9,1}$.

\begin{proposition}\label{P44} The VOA $\mathfrak{V}(\mathcal{S}(5,4,0))$ is isomorphic to the lattice VOA associated with the Niemeier lattice $N(A_{15}D_{9})$.
\end{proposition}

\noindent {\bf Case: $\mathcal{S}=\mathcal{S}(5,1,1)$.}
By Lemma \ref{LVS2} and Corollary \ref{Ex}, $\dim
\mathfrak{V}(\mathcal{S})_1=96$. By Propositions \ref{PDM1}, each simple
component $\mathfrak{g}_i$ of $\mathfrak{V}(\mathcal{S})_1$ satisfies
$h_i/k_i=3$. Hence by Proposition \ref{PDM2}, $\mathfrak{g}_i$ is one of  the
following.
\[
\begin{array} {l|c|c|c|c|c|c|c|c|c|c}
\text{Type} & A_{2,1}&A_{5,2}&A_{8,3}&C_{2,1}&B_{5,3}&C_{5,2}&D_{4,2}& D_{7,4}&E_{6,4}&F_{4,3}     \\     \hline
\ \ h^\vee      &  3 & 6 &9 & 3& 9 & 6& 6& 12& 12& 9 \\ \hline
\text{Dimension}& 8&35&80&10&55&55&28&91&78&52
\end{array}
\]
Take non-zero $t_i\in S_i$. Then by Lemma \ref{Lrank},
$H=\mathfrak{V}(\{(t_1,0,0),(0,t_2,0)\})_1$ is a Cartan subalgebra of
$\mathfrak{V}(\mathcal{S})_1$, and the rank is $16$. Consider the root space
decomposition with respect to $H$.  Then it is easy to see that
\begin{eqnarray*}
&&\mathfrak{V}(\{(y+s_1,y+s_2,0), (y+s_1,0,y),(0,y+s_2,y)\mid s_i\in S_i\})_1,\\
&&\mathfrak{V}(\{(y+a+s_1,y+a+s_2,0), (b+s_1,b+s_2,0)\mid s_i\in S_i,\ a\in S(P),\ b\in\overline{S(P)}\})_1,\\
&&\mathfrak{V}(\{(y+a+s_1,0, y+a), (b+s_1,0,b)\mid s_1\in S_1,\ a\in S(Q),\ b\in\overline{S(Q)}\})_1,\\
&&\mathfrak{V}(\{(0,y+a+s_2,y+\varphi(a)), (0,b+s_2,\varphi(b))\mid s_2\in S_2,\ a\in S(T),\ b\in\overline{S(T)}\})_1
\end{eqnarray*} are mutually orthogonal root spaces. Here we use the same notations as in Theorem
\ref{TClassify2}. Note that their dimensions are $8$, $12$, $30$, $30$. Hence
the $8$-dimensional and $12$-dimensional root spaces are $C_{2,1}$ and
$(A_{2,1})^2$, respectively. Note that the $30$-dimensional root space is
$(A_{2,1})^5$, $D_{4,2}A_{2,1}$ or $A_{5,2}$. Since the rank of
$\mathfrak{V}(\mathcal{S})_1$ is $16$, $\mathfrak{V}(\mathcal{S})_1$ must be
isomorphic to  $(A_{5,2})^2C_{2,1}(A_{2,1})^2$.

\begin{proposition}
The Lie algebra structure of $\mathfrak{V}(\mathcal{S}(5,1,1))_1$ is
 $(A_{5,2})^2C_{2,1}(A_{2,1})^2$.
\end{proposition}

\noindent {\bf Case: $\mathcal{S}=\mathcal{S}(5,3,1)$. }   By Lemma \ref{LVS2}
and Corollary \ref{Ex}, $\dim \mathfrak{V}(\mathcal{S})_1=240$. By Proposition
\ref{PDM1}, each simple component $\mathfrak{g}_i$ of
$\mathfrak{V}(\mathcal{S})_1$ satisfies $h_i/k_i=9$. Hence by Proposition
\ref{PDM2}, $\mathfrak{g}_i$ is one of  the following.
\[
\begin{array} {l|c|c|c|c|c|c}
\text{Type} & A_{8,1}&B_{5,1}&C_{8,1}&D_{10,2}&E_{7,2}&F_{4,1}   \\     \hline
\ \ h^\vee      &  9 & 9 & 9 & 18 &18 &9 \\ \hline
\text{Dimension}& 80&55&136&190&133&52
\end{array}
\]

Take non-zero $t_i\in S_i$.
Then by Lemma \ref{Lrank}, $\mathfrak{V}(\{(t_1,0,0),(0,t_2,0)\})_1$ is a Cartan subalgebra of $\mathfrak{V}(\mathcal{S})_1$, and the rank is $16$.
By the similar arguments as in Lemma \ref{LIL5}, $\mathfrak{V}(\mathcal{S})_1$
has rank $16$ and  has an ideal
$$\mathfrak{V}(\{(0,y+a+s_2,y+\varphi(a)), (0,b+s_2,\varphi(b))\mid s_2\in S_2,\ a\in S(T),\ b\in\overline{S(T)}\})_1$$
with $126$-dimensional root space.  Hence $\mathfrak{V}(\mathcal{S})_1$
contains $E_{7,2}$ and the Lie algebra structure of
$\mathfrak{V}(\mathcal{S})_1$ is isomorphic to  $E_{7,2}B_{5,1}F_{4,1}$.

\begin{proposition}
The Lie algebra structure of $\mathfrak{V}(\mathcal{S}(5,3,1))_1$ is
$E_{7,2}B_{5,1}F_{4,1}$.
\end{proposition}

\noindent {\bf Case: $\mathcal{S}=\mathcal{S}(5,2,2)$.}
By Lemma \ref{LVS2} and Corollary \ref{Ex}, $\dim
\mathfrak{V}(\mathcal{S})_1=192$. By Proposition \ref{PDM1}, each simple
component $\mathfrak{g}_i$ of $\mathfrak{V}(\mathcal{S})_1$ satisfies
$h_i/k_i=7$. Hence by Proposition \ref{PDM2}, $\mathfrak{g}_i$ is one of  the
following.
\[
\begin{array} {l|c|c|c|c}
\text{Type} & A_{6,1}&B_{4,1}&C_{6,1}&D_{8,2}     \\     \hline
\ \ h^\vee      &  7 & 7 &7 & 14 \\ \hline
\text{Dimension}& 48&36&78&120
\end{array}
\]
By Lemma \ref{Lrank}, the rank is $16$. Hence $\mathfrak{V}(\mathcal{S})_1$
must be isomorphic to $D_{8,2}(B_{4,1})^2$ or $(C_{6,1})^2B_{4,1}$. Moreover,
by Lemma \ref{Lrank2}, $\mathfrak{V}(\mathcal{S})_1$ has a subalgebra of type
$(A_{1,1})^{16}$. Hence it must be isomorphic to $(C_{6,1})^2B_{4,1}$.

\begin{proposition} The Lie algebra structure of $\mathfrak{V}(\mathcal{S}(5,2,2))_1$ is $(C_{6,1})^2B_{4,1}$.
\end{proposition}

\subsection{Isomorphism type of the VOA $\mathfrak{V}(\mathcal{S}(5,3,0,+))$}
In this subsection, we will show that the VOA $\mathfrak{V}(\mathcal{S}(5,3,0,+))$ is isomorphic to $\tilde{V}_{N(A_{15}D_{9})}$.
In order to do it, we use the $\Z_2$-orbifolds of VOAs associated to maximal totally singular subspaces of $R(V)^3$.

Let $\mathcal{S}$ be a maximal totally
singular subspace of $R(V)^3$. Then $\mathfrak{V}(\mathcal{S})$ is a holomorphic
VOA. Let $W\in R(V)^3\setminus \mathcal{S}$ with $q_V^3(W)=0$. Let
$\chi_{W}:\mathcal{S}\to\Z_2$ be the character of $\mathcal{S}$ defined by
$\chi_W(W')=\langle W,W'\rangle$.
Then $\chi_W$ induces the automorphism $g_W$ of $\mathfrak{V}(\mathcal{S})$ acting
on $M'$ by $(-1)^{\chi_W(W')}$ for $W'=[M']\in\mathcal{S}$.

\begin{proposition}\label{PZ2} The $\Z_2$-orbifold of $\mathfrak{V}(\mathcal{S})$ associated to $g_W$ is given by $\mathfrak{V}(\Span_{\F_2}\{W, \mathcal{S}\cap W^\perp\})$.
\end{proposition}
\begin{proof}
The subspace fixed by $g_W$ is $\mathfrak{V}(\mathcal{S}\cap W^\perp)$. By the
maximality of $\mathcal{S}$, the irreducible modules for
$\mathfrak{V}(\mathcal{S}\cap W^\perp)$ with integral weights are $\mathfrak{V}(\mathcal{S}\cap W^\perp)$,
$\mathfrak{V}(\mathcal{S}\setminus W^\perp)$ and
$\mathfrak{V}(W+(\mathcal{S}\cap W^\perp))$. Hence the $\Z_2$-orbifold of
$\mathfrak{V}(\mathcal{S})$ associated to $g_W$ is
$$\mathfrak{V}(\mathcal{S}\cap W^\perp)\oplus\mathfrak{V}(W+(\mathcal{S}\cap
W^\perp))=\mathfrak{V}(\Span_{\F_2}\{W, \mathcal{S}\cap W^\perp\})$$ as desired.
\end{proof}

Let us consider $\mathcal{S}=\mathcal{S}(5,4,0)$.
Then $\mathcal{S}=\Span_{\F_2}\{(s,0,0),(0,t,t),(y,y,0),(y,0,y),(z,z,z)\mid s\in S_1,t\in T\}$, where $T$ is an $8$-dimensional non-singular quadratic subspace of $R(V)$, $\Span_{\F_2} \{y,z\}$ is the orthogonal complement of $T$ in $R(V)$, and $S_1$ is a maximal totally singular subspace of $T$ (see Theorem \ref{TClassify2}).
Note that $q_V(y)=1$, $q_V(z)=0$ and $\langle y,z\rangle=1$.
Take $s_0\in S_1$ and $t_0\in T$ with $q_V(t_0)=0$ and $\langle s_0,t_0\rangle=1$.
Set $W=(t_0,0,z)$.
Then $\Span_{\F_2}\{ W,\mathcal{S}\cap W^\perp\}=\Span_{\F_2}\{(s,0,0),(0,t,t),(y,y,0),(t_0,0,z),(t_0+z,z,0),(s_0+y,0,y)\mid s\in S_1\cap t_0^\perp, t\in T\}$.
Since for $i=1,2$ and $j=1,3$, $\rho_i(\Span_{\F_2}\{(y,y,0),(t_0+z,z,0)\})$ and $\rho_j(\Span_{\F_2}\{(t_0,0,z),(s_0+y,0,y)\})$ are non-singular $2$-dimensional quadratic subspaces of plus type, $\Span_{\F_2}\{W,\mathcal{S}\cap W^\perp\}$ is conjugate to $\mathcal{S}(5,3,0,+)$ under $\Aut(V)\wr\Sym_3$.
Hence by Proposition \ref{PZ2}, $\mathfrak{V}(\mathcal{S}(5,3,0,+))$ is obtained by the $\Z_2$-orbifold of $\mathfrak{V}(\mathcal{S}(5,4,0))$ associated to $g_W$.

Recall from Proposition \ref{P44} that $\mathfrak{V}(\mathcal{S}(5,4,0))$ is isomorphic to the VOA associated to the Niemeier lattice $N(A_{15,1}D_{9,1})$.
Let us show that $g_W$ is conjugate to a lift of the $-1$-isometry of the lattice $N(A_{15,1}D_{9,1})$.
By \cite[Appendix D]{DGH}, it suffices to show that $g_W$ acts by $-1$ on a Cartan subalgebra of $\mathfrak{V}(\mathcal{S}(5,4,0))_1$.
Consider the subspace $\mathfrak{V}(\{(s_0,0,0),(0,s+y,s+y)\mid s\in S_1\})_1$ of $\mathfrak{V}(\mathcal{S}(5,4,0))_1$.
Then by Lemmas \ref{Lsemi} and \ref{LIL6}, it is a $24$-dimensional toral abelian subalgebra, that is, a Cartan subalgebra.
Since $\{(s_0,0,0),(0,s+y,s+y)\mid s\in S_1\}\cap W^\perp=\emptyset$, $g_W$ acts by $-1$ on this Cartan subalgebra.
Thus we obtain the following proposition.

\begin{proposition} The VOA $\mathfrak{V}(\mathcal{S}(5,3,0,+))$ is isomorphic to $\tilde{V}_{N(A_{15}D_{9})}$.
\end{proposition}

\paragraph{\bf Classification of Lie algebra structures }
By Lemmas \ref{LSY} and \ref{Lconj}, Theorem \ref{TC}, Proposition \ref{PS1} and Section 4.6, we
obtain the following theorem.

\begin{theorem}\label{e8+} Let $U$ be a holomorphic simple current extension of $(V_{\sqrt2E_8}^+)^{\otimes 3}$.
Then one of the following holds:
\begin{enumerate}
\item $U$ is isomorphic to a lattice VOA $V_N$ or its $\Z_2$-orbifold $\tilde{V}_N$ for some even unimodular lattice $N$;
\item The weight one subspace $U_1$ is isomorphic to one of the Lie algebras in Table \ref{Ta8}.
\end{enumerate}
\end{theorem}

\begin{table}[bht]\caption{Lie algebra structure of $\mathfrak{V}(\mathcal{S})_1$ for $\mathcal{S}\subset R(V)^3$}\label{Ta8}
$$\begin{array}{|c|c|c|c|c|}
\hline
{\mathcal{S}} & {\dim\mathfrak{V}(\mathcal{S})_1} &  \hbox{Lie algebra }&\hbox{No. in \cite{Sc93}}& \hbox{Ref.} \\ \hline
\mathcal{S}(5,1,0,+)& 60 & D_{4,4}(A_{2,2})^4&13& \hbox{New} \\ \hline
\mathcal{S}(5,1,0,-)&84 & C_{4,2}(A_{4,2})^2&22& \hbox{New} \\ \hline
\mathcal{S}(5,3,0,+)&192& D_{8,2}(B_{4,1})^2& 47& \tilde{V}_{N(A_{15}D_{9})} \\ \hline
\mathcal{S}(5,3,0,-)&240& C_{8,1}(F_{4,1})^2&52& \hbox{New} \\ \hline
\mathcal{S}(5,5,0,+)&744& D_{16,1}E_{8,1}& 69& V_{N(D_{16}E_{8})}\\ \hline
\mathcal{S}(5,2,1,+)&120& A_{7,2}(C_{3,1})^2A_{3,1}&33& \hbox{\cite{Lam}}\\ \hline
\mathcal{S}(5,2,1,-)&168& E_{6,2}C_{5,1}A_{5,1}&44&\hbox{New}\\ \hline
\mathcal{S}(5,4,1,+)&384& E_{8,2}B_{8,1}&62&\hbox{New}\\ \hline
\mathcal{S}(5,3,2,+)&240& C_{8,1}(F_{4,1})^2&52&\hbox{New}\\ \hline
\mathcal{S}(5,0,0)&48& (A_{3,4})^3A_{1,2}&7&\hbox{\cite{Lam}}\\ \hline
\mathcal{S}(5,2,0)&120& A_{7,2}(C_{3,1})^2A_{3,1}& 33&\hbox{\cite{Lam}}\\ \hline
\mathcal{S}(5,4,0)&408& A_{15,1}D_{9,1}& 63&V_{N(A_{15}D_{9})}\\ \hline
\mathcal{S}(5,1,1)&96&(A_{5,2})^2C_{2,1}(A_{2,1})^2 &26 &\hbox{\cite{Lam}}\\ \hline
\mathcal{S}(5,3,1)&240&E_{7,2}B_{5,1}F_{4,1}&53&\hbox{New}\\ \hline
\mathcal{S}(5,2,2)&192& (C_{6,1})^2B_{4,1}&48&\hbox{\cite{Lam}}\\ \hline
\end{array}$$
\end{table}

\section{Framed VOAs associated to subcodes of $\EuD(e_8)\oplus \EuD(d_{16}^+)$}

Recall that $\EuD(d_{16}^+)^\perp =\mathrm{Span}_{\Z_2} \{
d(\mathcal{E}_{16}), \ell(d_{16}^+)\}$ and the corresponding binary code VOA
is isomorphic to $V_{\sqrt{2}D_{16}^+}^+$. Note also that
$\EuD(e_8)^\perp \cong {\rm RM}(2,4)$ and $M_{{\rm RM}(2,4)}\cong V_{\sqrt{2}E_8}^+$.

Throughout this section, let $V=V_{\sqrt2E_8}^+$ and
$X=V_{\sqrt2D_{16}^+}^+$. For the detail of $V$ and $X$, see Sections 3.1 and
3.2, respectively. In this section, we study holomorphic VOAs associated to
maximal totally singular subspaces of $(R(X)\oplus R(V),q_X+q_V)$, which are
holomorphic simple current extensions of $X\otimes V$, and classify such VOAs.

\subsection{Simple current extensions of $V_{\sqrt2D_{16}^+}^+\otimes V_{\sqrt2E_8}^+$}
In this section, we study relations between simple current extensions of $X\otimes V$ and totally singular subspaces of $R(X)\oplus R(V)$.

We identify $R(X\otimes V)$ with $R(X)\oplus R(V)$ by Lemma \ref{LemFHL}.
By Lemmas \ref{POVE} (2) and \ref{POVE2} (2), $(R(X)\oplus R(V),q_X+q_V)$ is a non-singular $28$-dimensional quadratic space of plus type over $\F_2$.

\begin{notation}
Let $\mathcal{T}$ be a subset of $R(X)\oplus R(V)$.  Define
$\mathfrak{V}(\mathcal{T})=\oplus_{[M]\in\mathcal{T}}M$.
\end{notation}
The following proposition can be  shown by the same argument in \cite{Sh6} (cf.
Proposition \ref{PS1}).

\begin{proposition}\label{PS16} Let $V=V_{\sqrt2E_8}^+$ and $X=V_{\sqrt2D_{16}^+}^+$.
Then the $X\otimes V$-module
$\mathfrak{V}(\mathcal{T})=\oplus_{[M]\in\mathcal{T}}M$ has a simple VOA
structure which extends its $X\otimes V$-module structure if and only if
$\mathcal{T}$ is a totally singular subspace of $R(X)\oplus R(V)$. Moreover,
$\mathfrak{V}(\mathcal{T})$ is holomorphic if and only if $\mathcal{T}$ is
maximal.
\end{proposition}

\begin{remark} Let $\mathcal{T}$ be a totally singular subspace of $R(X)\oplus R(V)$.
Since $V$ and $X$ are framed, so is $\mathfrak{V}(\mathcal{T})$.
Hence $\mathfrak{V}(\mathcal{T})$ is simple, rational, $C_2$-cofinite and of CFT type (\cite{DGH}).
\end{remark}

Clearly $\Aut(X\otimes V)$ contains $\Aut(X)\times \Aut(V)$. By Lemma
\ref{LSY}, conjugates of $\mathfrak{V}(\mathcal{T})$ under $\Aut(X\otimes V)$
give isomorphic VOAs.

\begin{lemma}\label{Lknown} Let $\mathcal{S}$ be a maximal totally singular subspace of $R(X)\oplus R(V)$.
\begin{enumerate}
\item If $\mathcal{S}$ contains $(a_1,0), (0,a_2)$ for some $a_1\in\{[0]^-,[\alpha_1]^\pm\}\subset R(X)$ and $a_2\in S(R(V))^\times$ then $\mathfrak{V}(\mathcal{S})$ is isomorphic to a lattice VOA $V_L$.
\item If $\mathcal{S}$ contains $(a_1,a_2)$ for some $a_1\in\{[0]^-,[\alpha_1]^\pm\}\subset R(X)$ and $a_2\in S(R(V))^\times$ then $\mathfrak{V}(\mathcal{S})$ is isomorphic to $V_L$ or its $\Z_2$-orbifold $\tilde{V}_L$.
\item If $\mathcal{S}$ contains $([\alpha_c/2]^\varepsilon,0)$ with $\wt(c)=8$ then $\mathfrak{V}(\mathcal{S})$ contains a full subVOA isomorphic to $(V_{\sqrt2E_8}^+)^{\otimes 3}$.
\end{enumerate}
\end{lemma}
\begin{proof} Recall that the orbit of $[0]^-$ in $\Aut(X)$ is $\{[0]^-,[\alpha_1]^\pm\}$ (see Table \ref{Orbit}).
Hence (1) and (2) are shown by the same argument in Lemma \ref{PS2}.

If $\mathcal{S}$ contains $[\alpha_c/2]^\varepsilon$ with $\wt(c)=8$ then up to conjugation, we may assume that $([\alpha_c/2]^+,0)\in\mathcal{S}$.
Hence $\mathfrak{V}(\mathcal{S})$ contains a full subVOA isomorphic to $V_{\sqrt2E_8\oplus\sqrt2E_8}^+\otimes V_{\sqrt2E_8}^+$ by Lemma \ref{wt8}, which proves (3).
\end{proof}

\begin{lemma}\label{Lsemi2} Let $\mathcal{T}$ be a totally singular subspace of $R(X)\oplus R(V)$ and $[M]$ an element in $\mathcal{T}$ with $M_1\neq0$.
Then for any $a\in M_1$, $a_{(0)}$ is semisimple on $\mathfrak{V}(\mathcal{T})_1$.
\end{lemma}
\begin{proof} By the action of $\aut(X\otimes V) $ (cf. Table \ref{Orbit}), we may assume that $[M]$ is the tensor product of irreducible modules of untwisted type for $X$ and $V$.
Hence this lemma follows from \cite{FLM}.
\end{proof}

\subsection{Classification of maximal totally singular subspaces of $R(X)\oplus R(Y)$}
In this subsection, we study maximal totally singular subspaces of $R(X)\oplus R(V)$.
Let $\rho_1$ and $\rho_2$ denote the projections from $R(X)\oplus R(V)$ to $R(X)$ and to $R(V)$, respectively.
For a subset $\mathcal{S}$ of $R(X)\oplus R(V)$, let $\mathcal{S}^{(i)}=\{W\in \mathcal{S}\mid \rho_i(W)=0\}$.

\begin{lemma}\label{Lholo} Let $\mathcal{S}$ be a maximal totally singular subspace of $R(X)\oplus R(V)$.
Then the following hold:
\begin{enumerate}
\item For $\{i,j\}=\{1,2\}$, $\rho_i(\mathcal{S})=\rho_i(\mathcal{S}^{(j)})^\perp$;
\item There is a bijection from $\rho_1(\mathcal{S})/\rho_1(\mathcal{S}^{(2)})$ to $\rho_2(\mathcal{S})/\rho_2(\mathcal{S}^{(1)})$;
\item $\dim\rho_1(\mathcal{S}^{(2)})\ge 4$.
\end{enumerate}
\end{lemma}
\begin{proof} (1) follows from the maximality of $\mathcal{S}$.

Let $W^1\in \rho_1(\mathcal{S})$.
Then there is $W^2\in \rho_2(\mathcal{S})$ such that $(W^1,W^2)\in \mathcal{S}$.
By the definition of $\rho_i$, the map $W_1+\rho_1(\mathcal{S}^{(2)})\mapsto W_2+\rho_2(\mathcal{S}^{(1)})$ is a well-defined bijection, which proves (2).

Since $\mathcal{S}$ is maximal, $\dim\mathcal{S}=14$.
(3) follows from $\dim R(V)=10$.
\end{proof}

\begin{proposition}\label{PCl5} Let $\mathcal{S}$ be a maximal totally singular subspace of $R(X)\oplus R(V)$.
Assume that $\dim\rho_1(\mathcal{S}^{(2)})\ge 5$.
Then one of the following holds:
\begin{enumerate}
\item $\mathfrak{V}(\mathcal{S})$ contains a full subVOA isomorphic to $(V_{\sqrt2E_8}^+)^{\otimes 3}$;
\item $\mathfrak{V}(\mathcal{S})$ is isomorphic to a lattice VOA or its $\Z_2$-orbifold;
\item $\rho_1(\mathcal{S}^{(2)})$ is conjugate to $${\rm Span}_{\F_2}\{[\alpha_{(1^40^{12})}/2]^+, [\alpha_{(1^20^21^20^{10})}/2]^+, [\alpha_{((10)^40^8)}/2]^+, [\alpha_{(1^{16})}/4]^+,[\chi_0]^+\}.$$
\end{enumerate}
\end{proposition}
\begin{proof} Set $T=\rho_1(\mathcal{S}^{(2)})$ and $d=\dim T$.
Note that $T$ is a totally singular subspace of $R(X)$.
By the assumption, $\dim\rho_2(\mathcal{S})=14-d\le9$.
Hence $\dim\mathcal{S}^{(1)}\ge1$.
Up to the action of $\Aut(V_{\sqrt2E_8}^+)$ on the second coordinate, we may assume that $\mathcal{S}$ contains $([0]^+,[0]^-)$.

Set $T_1=\{[\lambda]^\pm\mid \lambda\in (\sqrt2D_{16}^+)^*/\sqrt2D_{16}^+\}\cap T$.
Then by Proposition \ref{fusion}, $T_1$ is a subspace of $T$, and $\dim T/T_1\le 1$, that is, $\dim T_1\ge d-1$.
If $T$ contains $[0]^-$ then (2) holds by the first paragraph and Lemma \ref{Lknown} (1).
Hence, we may assume $T$ does not contain $[0]^-$.

Let $L$ be the overlattice of ${\sqrt2D_{16}^+}$ such that $T_1=\{[\lambda]^{\varepsilon}\mid \lambda\in L/{\sqrt2D_{16}^+}\}$.
Since $T_1$ is totally singular, $L$ is even.
It follows from $\dim T_1\ge d-1$ that $|L/\sqrt2D_{16}^+|\ge 2^{d-1}$.
We now use the descriptions of $\sqrt2D_{16}^+$ and its dual lattice given in Section 3.2.
Then $L$ contains a sublattice $$L_1=\sum_{1\le i,j\le 16}\Z(\alpha_i+\alpha_j)+\sum_{c\in C}\Z(\frac{\alpha_c}{2}-\delta_c\alpha_1)$$ such that $L/L_1\subset{\rm Span}_{\F_2}\{\alpha_1,\alpha_{(1^{16})}/4-\alpha_d/2\}$, where $\delta_c\in\{0,1\}$, $d\in\mathcal{E}_{16}$ and $C$ is a doubly even code.
If $L/\sqrt2D_{16}^+$ contains $\alpha_1+\sqrt2D_{16}^+$ then up to the action of $\Aut(V_{\sqrt2D_{16}^+}^+)$ on $T$, $T$ contains $[0]^-$ by Table \ref{Orbit}, and (2) holds.
Hence we may assume that $|L/L_1|\le 2$, which implies $|L_1/\sqrt2D_{16}^+|\ge 2^{d-2}$, namely,
 $\dim C\ge d-1$.
If $\dim C\ge 5$ then $C$ has a weight $8$ codeword, and by Lemma \ref{Lknown} (3), (1) holds.
In particular, if $d\ge 6$ then (1) or (3) holds.

Assume $d=5$.
If $|T/T_1|=1$ or $|L/L_1|=1$ then $\dim C\ge6$, and (1) holds.
Hence we may assume that $T/T_1=\{[0]^+,[\chi_\lambda]^+\}$ and that $L/L_1=\{0,\alpha_{(1^{16})}/4-\alpha_d/2\}$.
By Lemma \ref{Lknown}, we also may assume that $C$ does not have weight $8$ codewords.
Hence $C$ is equivalent to ${\rm Span}_{\F_2}\{(1^40^{14}),(1^20^21^20^{10}),((10)^40^{8}), (1^{16})\}$.
Up to the action of lifts of $\Aut(\sqrt2D_{16}^+)$ to $\Aut(V_{\sqrt2D_{16}^+}^+)$,
$T$ is conjugate to $${\rm Span}_{\F_2}\{[\alpha_{(1^40^{12})}/2]^+,\ [\alpha_{(1^20^21^20^{10})}/2]^+,\ [\alpha_{((10)^40^8)}/2],^+\ [\alpha_{(1^{16})}/4]^+,\ [\chi_\lambda]^+\}.$$
Since $T$ is totally singular, $\langle \lambda,\lambda\rangle+\langle
v+\lambda,v+\lambda\rangle\in2\Z$ for all $v\in L$ by Proposition \ref{fusion}.
Hence $\langle\lambda,v\rangle\in\Z$ for all $v\in L$. Up to the action of
$\langle\lambda,\cdot \rangle\in{\rm Hom}(\sqrt2D_{16}^+,\Z_2) \subset
\Aut(V_{\sqrt2D_{16}^+}^+)$, we may assume that $\lambda=0$. Hence (3)
holds.
\end{proof}

\begin{proposition}\label{PCl4} Let $\mathcal{S}$ be a maximal totally singular subspace of $R(X)\oplus R(V)$.
Assume that $\dim\rho_1(\mathcal{S}^{(2)})=4$.
Then one of the following holds:
\begin{enumerate}
\item $V$ contains a full subVOA isomorphic to $(V_{\sqrt2E_8}^+)^{\otimes 3}$;
\item $V$ is isomorphic to a lattice VOA or its $\Z_2$-orbifold;
\item $\rho_1(\mathcal{S}^{(2)})$ is conjugate to ${\rm Span}_{\F_2}\{ [0]^-,[\alpha_1]^+, [\alpha_{(1^40^{12})}/2]^+,\ [\alpha_{(1^20^21^20^{10})}/2]^+\}$;
\item $\rho_1(\mathcal{S}^{(2)})$ is conjugate to ${\rm Span}_{\F_2}\{ [0]^-, [\alpha_{(1^40^{12})}/2]^+,\ [\alpha_{(1^20^21^20^{10})}/2]^+,\ [\alpha_{((10)^40^8)}/2]^+\}$;
\item $\rho_1(\mathcal{S}^{(2)})$ is conjugate to ${\rm Span}_{\F_2}\{ [0]^-, [\alpha_{(1^40^{12})}/2]^+,\ [\alpha_{(1^20^21^20^{10})}/2]^+,\ [\alpha_{(1^{16})}/4]^+\}$;
\item $\rho_1(\mathcal{S}^{(2)})$ is conjugate to ${\rm Span}_{\F_2}\{ [\alpha_{(1^40^{12})}/2]^+,\ [\alpha_{(1^20^21^20^{10})}/2]^+,\ [\alpha_{(1^{16})}/4]^+, [\chi_0]^+\}$.

\end{enumerate}
\end{proposition}
\begin{proof} Set $T=\rho_1(\mathcal{S}^{(2)})$.
If $T$ contains $[\alpha_c/2]^\pm$ with $\wt(c)=8$ then (1) holds by Lemma \ref{Lknown} (3).
Hence we may assume that $T$ does not contain such elements.
If $T$ contains neither $[0]^-$ nor $[\chi_\lambda]^+$ then $T^\perp\setminus T$ contains $[0]^-$ by Lemmas \ref{Inner} and \ref{Lholo} (1).
In this case, (2) holds up to conjugation.
Hence we may assume that $T$ contains $[0]^-$ or $[\chi_\lambda]^+$.

Assume that $T$ contains $[0]^-$.
Then by Lemma \ref{Inner} $T$ is a subset of $$\{[0]^\pm,[\alpha_1]^\pm,[\alpha_c/2]^\pm,[\alpha_c/2-\alpha_i]^\pm,[\alpha_{(1^{16})}/4-\alpha_d/2]^\pm\mid \wt(c)=4,\ d\in\mathcal{E}_{16}\}.$$
By Lemma \ref{Inner}, $T$ can not contain both $[\alpha_1]^\pm$ and $[\alpha_{(1^{16})}/4-\alpha_d/2]^\pm$.
Thus we obtain one of (3), (4) and (5) by Proposition \ref{fusion} and Table \ref{Orbit}.

Assume that $T$ contains $[\chi_\lambda]^+$.
Up to the action of $\Aut(V_{\sqrt2D_{16}^+}^+)$, we may assume that $\lambda=0$.
If $T$ does not contain $[\alpha_{(1^{16})}/4-\alpha_d/2]^\pm$ then $g\circ T$ does not contain $[\chi_\lambda]^+$ for some $g\in\Aut(V_{\sqrt2D_{16}^+}^+)$ by Proposition \ref{fusion} and Table \ref{Orbit}, which is the case in the previous paragraph.
Hence we may assume that $T$ contains $[\alpha_{(1^{16})}/4-\alpha_d/2]^+$.
By the action of automorphisms induced from $\Aut(\sqrt2D_{16}^+)$, we may assume $d=0$.
By Lemma \ref{Inner}, $T$ does not contain $[\alpha_1]^+$.
Thus we have (6).
\end{proof}

Let us calculate the dimension of the weight $1$ subspace of $\mathfrak{V}(\mathcal{S})$.

\begin{proposition}\label{Pdim16} Let $\mathcal{S}$ be a maximal totally singular subspace of $R(X)\oplus R(V)$.
Then the dimension of $\mathfrak{V}(\mathcal{S})_1$ is given by
\begin{eqnarray}
 16\times|\rho_1(\mathcal{S}^{(2)})\cap\{[0]^-,[\alpha_1]^\pm\}|+4\times|\rho_1(\mathcal{S}^{(2)})\cap\{[\alpha_c/2]^\pm,\ [\alpha_c/2-\alpha_1]^\pm\mid \wt(c)=4\}|\notag &&\\
+|\rho_1(\mathcal{S}^{(2)})\cap\{[\alpha_{(1^{16})}/4-\alpha_d/2]^\pm,\ [\chi_\lambda]^+\mid d\in\mathcal{E}_{16}\}|+8\times(|\rho_2(\mathcal{S}^{(1)})|-1)&&\label{Eq:Pdim}\\
+|\rho_1(\mathcal{S})\cap\{[\alpha_c/2-\alpha_1]^\pm,\ [\alpha_c/2]^\pm\mid \wt(c)=2\}|\times|\rho_2(\mathcal{S}^{(1)})|.&&\notag
\end{eqnarray}
\end{proposition}
\begin{proof} Let $(W^1,W^2)\in\mathcal{S}$.
Then the lowest weights of $W^i$ are non-negative half integers, and the sum of the lowest weights is an integer.
Hence $\mathfrak{V}(\{(W^1,W^2)\})_1\neq0$ if and only if the lowest weights of both $W^1$ and $W^2$ are $1/2$, or the lowest weight of $W^i$ is $1$ and $W^j=[0]^+$, where $\{i,j\}=\{1,2\}$.
By Table \ref{Orbit}, Lemmas \ref{POVE} (2) and \ref{Lholo}, the dimension of the weight $1$ subspace of $\mathfrak{V}(\mathcal{S})$ is given by (\ref{Eq:Pdim}).
\end{proof}

\subsection{Determination of the Lie algebra structure of $\mathfrak{V}(\mathcal{S})_1$ for $\mathcal{S}\subset R(X)\oplus R(V)$}
In this subsection, we determine the Lie algebra structure  of
$\mathfrak{V}(\mathcal{S})_1$ for a maximal totally singular subspace
$\mathcal{S}$ of $R(X)\oplus R(V)$ described in Propositions \ref{PCl5} (3) and
\ref{PCl4} (3)--(6).

\noindent {\bf Case: $\mathcal{S}$ as in Propositions \ref{PCl5} (3).}

\begin{proposition} Let $\mathcal{S}$ be a maximal totally singular subspace of $R(X)\oplus R(V)$.
Assume that $\rho_1(\mathcal{S}^{(2)})={\rm Span}_{\F_2}\{[\alpha_{(1^40^{12})}/2]^+,\ [\alpha_{(1^20^21^20^{10})}/2]^+,\ [\alpha_{((10)^40^8)}/2]^+,\ [\alpha_{(1^{16})}/4]^+,\ [\chi_0]^+\}$.
Set $D={\rm Span}_{\F_2}\{(1^40^{12}), (1^20^21^20^{10}), ((10)^40^8)\}$.
Then the following hold:
\begin{enumerate}
\item
$\rho_1(\mathcal{S})\cap\{[\alpha_c/2-\alpha_1]^\pm,\ [\alpha_c/2]^\pm \mid \wt(c)=2\}=\{[\alpha_{c}/2-\alpha_1]^+\mid \wt(c)=2,\ c\in D^\perp \}$ and its size is $36$;
\item $\dim \mathfrak{V}(\mathcal{S})_1=132$, and the Lie algebra structure of $\mathfrak{V}(\mathcal{S})_1$ is $A_{8,2}F_{4,2}$.
\end{enumerate}
\end{proposition}
\begin{proof} Note that the number of codewords in $D^\perp$ with weight $2$ is $\binom{9}{2}=36$.
(1) is easily calculated by Lemmas \ref{Inner} and \ref{Lholo} (1).

By Proposition \ref{Pdim16}, we obtain $$\dim
\mathfrak{V}(\mathcal{S})_1=16\times 0+4\times 7+24+8\times 1+36\times
2=132.$$ By Proposition \ref{PDM1}, each simple component $\mathfrak{g}_i$
of $\mathfrak{V}(\mathcal{S})_1$ satisfies $h_i/k_i=9/2$. Hence by Proposition
\ref{PDM2} , $\mathfrak{g}_i$ is one of
\[
\begin{array} {l|c|c|c}
\text{Type} & A_{8,2}&B_{5,2} & F_{4,2}    \\     \hline
\ \ h^\vee      &  9 & 9 &9 \\ \hline
\text{Dimension}& 80&55&52
\end{array}
\]
Hence the Lie algebra structure is $A_{8,2}F_{4,2}$.
\end{proof}

\noindent {\bf Case:  $\mathcal{S}$ as in Proposition \ref{PCl4} (3).}
\begin{proposition} Let $\mathcal{S}$ be a maximal totally singular subspace of $R(X)\oplus R(V)$.
Assume that $\rho_1(\mathcal{S}^{(2)})={\rm Span}_{\F_2}\{ [0]^-,\ [\alpha_1]^+,\ [\alpha_{(1^40^{12})}/2]^+,\ [\alpha_{(1^20^21^20^{10})}/2]^+\}$ holds.
Set $D={\rm Span}_{\F_2}\{(1^40^{12}), (1^20^21^20^{10})\}$.
Then the following hold:
\begin{enumerate}
\item $\rho_1(\mathcal{S})\cap\{[\alpha_c/2-\alpha_1]^\pm,\ [\alpha_c/2]^\pm\mid \wt(c)=2\}=\{[\alpha_c/2]^\pm,\ [\alpha_{c}/2-\alpha_1]^\pm\mid \wt(c)=2,\ c\in D^\perp\}$ and its size is $192$;
\item $\dim \mathfrak{V}(\mathcal{S})_1=288$, and the Lie algebra structure of $\mathfrak{V}(\mathcal{S})_1$ is $C_{10,1}B_{6,1}$.
\end{enumerate}
\end{proposition}
\begin{proof}
Note that the number of codewords in $D^\perp$ with weight $2$ is $\binom{10}{2}+3=48$.
(1) is easily calculated by Lemmas \ref{Inner} and \ref{Lholo} (1).

By Proposition \ref{Pdim16}, we have $$\dim
\mathfrak{V}(\mathcal{S})_1=16\times 3+4\times 12+0+8\times0+192\times 1=288.$$
By Proposition \ref{PDM1}, each simple component $\mathfrak{g}_i$ of
$\mathfrak{V}(\mathcal{S})_1$ satisfies $h_i/k_i=11$. Hence by Proposition
\ref{PDM2}, $\mathfrak{g}_i$ is one of the following.
\[
\begin{array} {l|c|c|c}
\text{Type} & A_{10,1}&B_{6,1} & C_{10,1}    \\     \hline
\ \ h^\vee      &  11 & 11 &11 \\ \hline
\text{Dimension}& 120&78&210
\end{array}
\]
Hence the Lie algebra structure is $C_{10,1}B_{6,1}$.
\end{proof}

\noindent {\bf Case:  $\mathcal{S}$ as in Proposition \ref{PCl4} (4).}
\begin{proposition}\label{PCl44} Let $\mathcal{S}$ be a maximal totally singular subspace of $R(X)\oplus R(V)$.
Assume that $\rho_1(\mathcal{S}^{(2)})={\rm Span}_{\F_2}\{ [0]^-, [\alpha_{(1^40^{12})}/2]^+,\ [\alpha_{(1^20^21^20^{10})}/2]^+,\ [\alpha_{((10)^40^8)}/2]^+\}$.
Set $D={\rm Span}_{\F_2}\{(1^40^{12}), (1^20^21^20^{10}), (10)^40^8)\}$.
Then the following hold:
\begin{enumerate}
\item
$\rho_1(\mathcal{S})\cap\{[\alpha_c/2-\alpha_1]^\pm,\ [\alpha_c/2]^\pm\mid \wt(c)=2\}=\{[\alpha_c/2]^\pm,\ [\alpha_{c}/2-\alpha_1]^\pm\mid \wt(c)=2,\ c\in D^\perp \}$ and its size is $144$;
\item $\dim \mathfrak{V}(\mathcal{S})_1=216$, and the Lie algebra structure of $\mathfrak{V}(\mathcal{S})_1$ is $D_{9,2}A_{7,1}$.
\end{enumerate}
\end{proposition}
\begin{proof}
Note that the number of codewords in $D^\perp$ with weight $2$ is $\binom{10}{2}+3=48$.
(1) is easily calculated by Lemmas \ref{Inner} and \ref{Lholo} (1).

By Proposition \ref{Pdim16}, we have $$\dim
\mathfrak{V}(\mathcal{S})_1=16\times 1+4\times 14+0+8\times0+144\times 1=216.$$
By Proposition \ref{PDM1}, each simple component $\mathfrak{g}_i$ of
$\mathfrak{V}(\mathcal{S})_1$ satisfies $h_i/k_i=8$. Hence by Proposition
\ref{PDM2}, $\mathfrak{g}_i$ is one of the following.
\[
\begin{array} {l|c|c|c|c}
\text{Type} & A_{7,1}&C_{7,1}&D_{5,1}&D_{9,2} \\     \hline
\ \ h^\vee      &  8 & 8 &8 & 16 \\ \hline
\text{Dimension}& 63&105&45&153
\end{array}
\]
Hence $\mathfrak{V}(\mathcal{S})_1$ is $(A_{7,1})^2(D_{5,1})^2$ or
$D_{9,2}A_{7,1}$.
Clearly, $\mathfrak{V}(\{([0]^-,[0]^+)\})_1$ is an abelian subalgebra of $\mathfrak{V}(\mathcal{S})$.
By the fusion rules $[0]^-\times[\lambda]^\pm=[\lambda]^\mp$ and $\mathcal{S}^{(1)}=0$, it is maximal abelian.
By Lemma \ref{Lsemi2}, it is toral. Hence the rank of
$\mathfrak{V}(\mathcal{S})_1$ is $16$. Thus the Lie algebra structure of
$\mathfrak{V}(\mathcal{S})_1$ is $D_{9,2}A_{7,1}$.
\end{proof}
\noindent {\bf Case:  $\mathcal{S}$ as in Proposition \ref{PCl4} (5).}
\begin{proposition} Let $\mathcal{S}$ be a maximal totally singular subspace of $R(X)\oplus R(V)$.
Assume that $\rho_1(\mathcal{S}^{(2)})={\rm Span}_{\F_2}\{ [0]^-, [\alpha_{(1^40^{12})}/2]^+,\ [\alpha_{(1^20^21^20^{10})}/2]^+,\ [\alpha_{(1^{16})}/4]^+\}$.
Set $D={\rm Span}_{\F_2}\{(1^40^{12}), (1^20^21^20^{10})\}$.
Then the following hold:
\begin{enumerate}
\item
$\rho_1(\mathcal{S})\cap\{[\alpha_c/2-\alpha_1]^\pm,\ [\alpha_c/2]^\pm\mid \wt(c)=2\}=\{[\alpha_{c}/2-\alpha_1]^\pm\mid \wt(c)=2,\ c\in D^\perp\}$ and its size is $96$;
\item $\dim \mathfrak{V}(\mathcal{S})_1=144$, and the Lie algebra structure is $A_{9,2}A_{4,1}B_{3,1}$.
\end{enumerate}
\end{proposition}
\begin{proof}
Note that the number of codewords in $D^\perp$ with weight $2$ is $\binom{10}{2}+3=48$.
(1) is easily calculated by Lemmas \ref{Inner} and \ref{Lholo} (1).

By Proposition \ref{Pdim16}, we have $$\dim
\mathfrak{V}(\mathcal{S})_1=16\times 1+4\times 6+8+8\times0+96\times1=144.$$ By
Proposition \ref{PDM1}, each simple component $\mathfrak{g}_i$ of
$\mathfrak{V}(\mathcal{S})_1$ satisfies $h_i/k_i=5$. Hence by Proposition
\ref{PDM2} $\mathfrak{g}_i$ is one of the following.
\[
\begin{array} {l|c|c|c|c|c|c}
\text{Type} & A_{4,1}&A_{9,2}&B_{3,1}&B_{8,3} & C_{4,1}&D_{6,2}   \\     \hline
\ \ h^\vee  &  5& 10& 5 &15 &5 & 10   \\ \hline
\text{Dimension}& 24&99&21&136&36 & 66
\end{array}
\]
By the same arguments as in Proposition \ref{PCl44}, $H=\mathfrak{V}(\{([0]^-,[0]^+)\})_1$ is a Cartan subalgebra of $\mathfrak{V}(\mathcal{S})_1$.
We decomposes $\mathfrak{V}(\mathcal{S})_1$ into a direct sum of root spaces for $H$.
Then $\mathfrak{V}(\mathcal{S})_1$ has an ideal with $90$-dimensional root space $$\mathfrak{V}(\{([\alpha_{c}/2-\alpha_1]^\varepsilon,W(c,\varepsilon))\mid \varepsilon\in\{\pm\},\ c\in D^\perp,\ \wt(c)=2,\ {\rm supp}(c)\cap{\rm supp}(D)=\emptyset\})_1,$$
where $W(c,\varepsilon)$ is a unique element in $R(V)$ such that $([\alpha_{c}/2-\alpha_1]^\varepsilon,W(c,\varepsilon))\in\mathcal{S}$.
Hence the root space of $\mathfrak{V}(\mathcal{S})_1$ has the decomposition $90+38$, and the Lie algebra structure is $A_{9,2}A_{4,1}B_{3,1}$ or $A_{4,1}(B_{3,1})^6$ or $A_{4,1}^3(B_{3,1})^2C_{4,1}$.
Since the rank of $\mathfrak{V}(\mathcal{S})_1$ is $16$, it must be $A_{9,2}A_{4,1}B_{3,1}$.
\end{proof}
\noindent {\bf Case: Proposition \ref{PCl4} (6).}
\begin{proposition} Let $\mathcal{S}$ be a maximal totally singular subspace of $R(X)\oplus R(V)$.
Assume that $\rho_1(\mathcal{S}^{(2)})={\rm Span}_{\F_2}\{ [\alpha_{(1^40^{12})}/2]^+,\ [\alpha_{(1^20^21^20^{10})}/2]^+,\ [\alpha_{(1^{16})}/4]^+, [\chi_0]^+\}$.
Set $D={\rm Span}_{\F_2}\{(1^40^{12}), (1^20^21^20^{10})\}$.
Then the following hold:
\begin{enumerate}
\item
$\rho_1(\mathcal{S})\cap\{[\alpha_c/2-\alpha_1]^\pm,\ [\alpha_c/2]^\pm\mid \wt(c)=2\}=\{[\alpha_{c}/2-\alpha_1]^+\mid \wt(c)=2,\ c\in D^\perp\}$ and its size is $48$;
\item $\dim \mathfrak{V}(\mathcal{S})_1=72$, and the Lie algebra structure is $D_{5,4}C_{3,2}(A_{1,1})^2$.
\end{enumerate}
\end{proposition}
\begin{proof}
Note that the number of codewords in $D^\perp$ with weight $2$ is $\binom{10}{2}+3=48$.
(1) is easily calculated by Lemmas \ref{Inner} and \ref{Lholo} (1).

By Proposition \ref{Pdim16}, we have $$\dim
\mathfrak{V}(\mathcal{S})_1=16\times0+4\times 3+12+8\times0+48\times 1=72.$$ By
Proposition \ref{PDM1}, each simple component $\mathfrak{g}_i$ of
$\mathfrak{V}(\mathcal{S})_1$ satisfies $h_i/k_i=2$. Hence by Proposition
\ref{PDM2}, $\mathfrak{g}_i$ is one of the following.
\[
\begin{array} {l|c|c|c|c|c|c|c|c|c|c}
\text{Type} & A_{1,1}&A_{3,2}&A_{5,3}&A_{7,4}&C_{3,2}&C_{5,3}
              &D_{4,3}&D_{5,4}&D_{6,5}&G_{2,2} \\     \hline
\ \ h^\vee      &  2 & 4 &6 & 8&4& 6 &6 &8 &10 &4 \\ \hline
\text{Dimension}& 3&15&35&63 & 21&55&28&45&66&14
\end{array}
\]
By Proposition \ref{fusion},  both
$$\mathfrak{V}(\{([\alpha_{c}/2-\alpha_1]^+,W(c))\mid \wt(c)=2,\ c\in D^\perp,\
{\rm supp}(c)\cap{\rm supp}(D)=\emptyset\})_1$$ and
$$\mathfrak{V}(\mathcal{S}^{(2)}\cup\{([\alpha_{c}/2-\alpha_1]^+,W(c))\mid
\wt(c)=2, c\in D^\perp,\ {\rm supp}(c)\subset{\rm supp}(D)\})_1$$ are ideals,
where $W(c)$ is a unique element in $R(V)$ such that
$([\alpha_{c}/2-\alpha_1]^+,W(c))\in\mathcal{S}$. Hence
$\mathfrak{V}(\mathcal{S})$ has the decomposition $45+27$.
Moreover, by the previous case, the first subspace is the fixed points of the Lie algebra of type $A_{9,2}$ by an order $2$ automorphism acting by $-1$ on the Cartan subalgebra.
Hence $\mathfrak{V}(\mathcal{S})_1$ contains $D_{5}$ as an ideal, and the Lie algebra structure is $D_{5,4}(A_{1,1})^9$, $D_{5,4}A_{3,2}(A_{1,1})^4$ or $D_{5,4}C_{3,2}(A_{1,1})^2$.
Since $\mathfrak{V}({\rm Span}_{\F_2}\{ [\alpha_{(1^40^{12})}/2]^+,\ [\alpha_{(1^20^21^20^{10})}/2]^+,\ [\alpha_{(1^{16})}/4]^+\})$ is isomorphic to $V_{A_3\oplus A_4}^+$, $\mathfrak{V}(\mathcal{S})_1$ contains a Lie subalgebra $C_{2,2} (A_{1,2})^2$.
Hence the Lie algebra structure is $D_{5,4}C_{3,2}(A_{1,1})^2$.
\end{proof}

\subsection{Isomorphism type of  $\mathfrak{V}(\mathcal{S})$ for $S$ in Proposition \ref{PCl4} (4)}
In this subsection, we show that the VOA associated to the maximal totally
singular space given in Proposition \ref{PCl4} (4) is isomorphic to
$\tilde{V}_{N(A_{17}E_{7})}$ as a VOA.

First, we construct a lattice VOA as a simple current extension of $X\otimes V$.

\begin{lemma} Let $\mathcal{S}$ be a maximal totally singular subspace of $R(X)\oplus R(V)$
such that $$\rho_1(\mathcal{S}^{(2)})={\rm Span}_{\F_2}\{ [0]^-,
[\alpha_1]^+,[\alpha_{c}/2]^+\mid c\in D\},$$ where $D={\rm Span}_{\F_2}\{(1^40^{12}), (1^20^21^20^{10}), ((10)^40^8)\}$. Then we have the following:
\begin{enumerate}
\item $\rho_1(\mathcal{S})\cap\{[\alpha_c/2-\alpha_1]^\pm,\ [\alpha_c/2]^\pm\mid \wt(c)=2\}=\{[\alpha_c/2]^\pm,\ [\alpha_{c}/2-\alpha_1]^\pm\mid \wt(c)=2,\ c\in D^\perp \}$ and its size is $144$;
\item $\mathfrak{V}(\mathcal{S})$ is isomorphic to a lattice VOA associated to $N(A_{17}E_{7})$ or $N(D_{10}E_{7}^2)$.
\end{enumerate}
\end{lemma}
\begin{proof} Note that the number of codewords in $D^\perp$ with weight $2$ is $\binom{10}{2}+3=48$.
(1) is easily calculated by Lemmas \ref{Inner} and \ref{Lholo} (1).

It follows from $\dim\rho_1(\mathcal{S}^{(2)})=5$ that $\dim\rho_2(\mathcal{S}^{(1)})=1$.
Up to conjugation by $\Aut(X\otimes V)$, we may assume that $\rho_2(\mathcal{S}^{(1)})=\{[0]^+,[0]^-\}$.
By Proposition \ref{Pdim16}, we have $$\dim \mathfrak{V}(\mathcal{S})_1=16\times 3+4\times 28+0+8\times 1+144\times 2=456.$$
By Proposition \ref{PDM1}, each simple component $\mathfrak{g}_i$ of $\mathfrak{V}(\mathcal{S})_1$ satisfies $h_i/k_i=18$.
Hence by Proposition \ref{PDM2}, $\mathfrak{g}_i$ is one of
\[
\begin{array} {l|c|c|c}
\text{Type} & A_{17,1}&D_{10,1}&E_{7,1} \\     \hline
\ \ h^\vee      &  18 & 18 &18 \\ \hline
\text{Dimension}& 323&190&133
\end{array}
\]
Hence $\mathfrak{V}(\mathcal{S})_1$ is a Lie algebra of type $A_{17,1}E_{7,1}$ or $D_{10,1}(E_{7,1})^2$.
\end{proof}

Let $\mathcal{S}$ be a maximal totally singular subspace of $R(X)\oplus R(V)$ in the lemma above and let $W=([\alpha_{(1^{16})}/4-\alpha_i]^+,[\chi_0]^+)$.
Then $\Span_{\F_2}\{W,\mathcal{S}\cap W^\perp\}$ satisfies Proposition \ref{PCl4} (4).
Hence by the same arguments as in Proposition \ref{PZ2}, $\mathfrak{V}(\Span_{\F_2}\{W,\mathcal{S}\cap W^\perp\})$ is obtained by the $\Z_2$-orbifold of $\mathfrak{V}(\mathcal{S})$ associated to $g_W$.

By the lemma above, $\mathfrak{V}(\mathcal{S})$ is isomorphic to the VOA associated to some even unimodular lattice of rank $24$.
Let us show that $g_W$ is conjugate to a lift of the $-1$-isometry of the lattice.
By \cite[Appendix D]{DGH}, it suffices to show that $g_W$ acts by $-1$ on a Cartan subalgebra of $\mathfrak{V}(\mathcal{S})_1$.
Consider the subspace $\mathfrak{V}(\{([\alpha_1]^+,[0]^+),([0]^+,[0]^-)\})_1$ of $\mathfrak{V}(\mathcal{S})_1$.
Then by Lemma \ref{Lsemi2}, it is a $24$-dimensional toral abelian subalgebra, that is, a Cartan subalgebra.
Since $\{([\alpha_1]^+,[0]^+),([0]^+,[0]^-)\}\cap W^\perp=\emptyset$, $g_W$ acts by $-1$ on this Cartan subalgebra.
Recall that $\tilde{V}_{N(D_{10}E_{7}^2)}\cong V_{N(D_5^2A_7^2)}$.
Hence $\mathfrak{V}(\mathcal{S})$ must be isomorphic to ${\tilde{V}}_{N(A_{17}E_{7})}$.
Thus we obtain the following proposition.

\begin{proposition} The VOA $\mathfrak{V}(\mathcal{S})$ associated to a maximal totally singular subspace $\mathcal{S}$ of $R(X)\oplus R(V)$ satisfying Proposition \ref{PCl4} (4) is isomorphic to $\tilde{V}_{N(A_{17}E_{7})}$.
\end{proposition}

\paragraph{\bf Classification of the Lie algebra structures}
As a summary of this section, we obtain the following theorem.

\begin{theorem}\label{d16+} Let $U$ be a holomorphic simple current extension of $V_{\sqrt2D_{16}^+}^+\otimes V_{\sqrt2E_{8}}^+$.
Then one of the following holds:
\begin{enumerate}
\item $U$ is isomorphic to a lattice VOA $V_N$ or its $\Z_2$-orbifold $\tilde{V}_N$ for some even unimodular lattice $N$;
\item $U$ contains $(V_{\sqrt2E_8}^+)^{\otimes 3}$ as a full subVOA;
\item The weight one space $U_1$ is isomorphic to one of the Lie algebras in Table \ref{Ta16}.
\end{enumerate}
\end{theorem}

\begin{table}[bht]
\caption{Lie algebra structure of $\mathfrak{V}(\mathcal{S})_1$ for $\mathcal{S}\subset R(X)\oplus R(V)$}\label{Ta16}
$$\begin{array}{|c|c|c|c|c|}
\hline
{\mathcal{S}} & {\dim\mathfrak{V}(\mathcal{S})} &  \hbox{Lie algebra }&\hbox{No. in \cite{Sc93}}&\hbox{Ref.} \\ \hline
\text{Proposition }\ref{PCl5} (3)& 132 & A_{8,2}F_{4,2}&36& \hbox{New} \\ \hline
\text{Proposition } \ref{PCl4} (3)&288 & C_{10,1}B_{6,1}&56& \hbox{\cite{Lam}} \\ \hline
\text{Proposition } \ref{PCl4} (4)&216& D_{9,2}A_{7,1}& 50&\tilde{V}_{N(A_{17}E_{7})} \\ \hline
\text{Proposition } \ref{PCl4} (5)&144& A_{9,2}A_{4,1}B_{3,1}&40& \hbox{\cite{Lam}} \\ \hline
\text{Proposition } \ref{PCl4} (6)&72& D_{5,4}C_{3,2}(A_{1,1})^2&19& \hbox{\cite{Lam}}\\ \hline
\end{array}$$
\end{table}

Finally, by combining Theorems \ref{VN}, \ref{Dex}, \ref{e8+} and \ref{d16+}, we
obtain our main theorem --Theorem \ref{thm:Lieframed}.

\begin{remark}
In \cite{HS}, it is announced that holomorphic framed VOAs having Lie algebra $A_{8,2}F_{4,2}$, $C_{4,2}A_{4,2}^2$ and $D_{4,4}A_{2,2}^4$ would be constructed as simple current extensions of $V_{\sqrt2E_8\oplus\sqrt2D_4}\otimes V_{\sqrt2D_{12}^+}^+$.
\end{remark}

\paragraph{\bf Acknowledgements.} The authors thanks to Professor Koichi Betsumiya and Professor Akihiro Munemasa for useful discussions on triply even codes.
Part of the work was done when the second author was visiting Academia Sinica in 2010 and Imperial College London in 2010.
He thanks the staff for their help.

\end{document}